\documentclass{amsart}

\usepackage{amsmath}
\usepackage{amsxtra}
\usepackage{amscd}
\usepackage{amsthm}
\usepackage{amsfonts}
\usepackage{amssymb}
\usepackage{eucal}
\usepackage[all]{xy}

\newtheorem{df}{Definition}[section]

\newtheorem{thm}[df]{Theorem}
\newtheorem{lem}[df]{Lemma}
\newtheorem{claim}[df]{Claim}
\newtheorem{conc}[df]{Conclusion}
\newtheorem{obs}[df]{Observation}

\newtheorem{nota}[df]{Notation}
\newtheorem{assu}[df]{Assumption}
\newtheorem{rem}[df]{Remark}

\newtheorem{conv}[df]{Convention}

\def\gP{{\frak P}}

\def\gL{{\frak L}}

\def\gH{{\frak H}}

\def\bN{{\mathbb N}}

\def\bp{{\bar p}}
\def\bq{{\bar q}}

\def\dcup{{\cup {\!{\!{\!{\cdot}}}}\,}}

\title{Hereditary zero-one laws for graphs}
\author{Saharon Shelah and Mor Doron}

\thanks{The authors would like to thank the Israel Science Foundation for partial support of this research (Grant no. 242/03). Publication no. 953 on Saharon Shelah's list.}

\begin{document}

\begin{abstract}
We consider the random graph $M^n_{\bar{p}}$ on the set $[n]$, were the probability of $\{x,y\}$ being an edge is $p_{|x-y|}$, and
$\bar{p}=(p_1,p_2,p_3,...)$ is a series of probabilities. We consider the set of all $\bar{q}$ derived from $\bar{p}$ by inserting 0 probabilities to $\bar{p}$, or alternatively by decreasing some of the $p_i$. We say that $\bar{p}$ hereditarily satisfies
the 0-1 law if the 0-1 law (for first order logic) holds in $M^n_{\bar{q}}$ for any $\bar{q}$ derived from $\bar{p}$ in the relevant way described above. We give a necessary and sufficient condition on $\bar{p}$ for it to hereditarily satisfy the 0-1 law.
\end{abstract}

\maketitle

\section{Introduction}
In this paper we will investigate the random graph on the set $[n]=\{1,2,...,n\}$ were the probability of a pair $i\neq j\in[n]$ being connected by an edge depends only on their distance $|i-j|$. Let us define:
\begin{df} \label{DefM}
For a sequence $\bp=(p_1,p_2,p_3,...)$ where each $p_i$ is a probability i.e. a real in $[0,1]$,
let $M^n_{\bp}$ be the random graph defined by:
\begin{itemize}
\item The set of vertices is $[n]=\{1,2,...,n\}$.
\item For $i,j\leq n$, $i\neq j$ the probability of $\{i,j\}$ being an
edge is $p_{|i-j|}$.
\item All the edges are drawn independently.
\end{itemize}
\end{df}

If $\gL$ is some logic, we say that $M^n_{\bp}$ satisfies the 0-1 law for the logic $\gL$ if for each sentence $\psi\in\gL$ the probability that $\psi$ holds in $M^n_{\bp}$ tends to $0$ or $1$, as $n$ approaches $\infty$. The relations between properties of $\bp$ and the asymptotic behavior of $M^n_{\bp}$ were investigated in \cite{LuSh}. It was proved there that for $L$, the first order logic in the vocabulary with only the adjacency relation, we have:

\begin{thm} \label{ThmPrevL}
\begin{enumerate}
\item Assume $\bp=(p_1,p_2,...)$ is such that $0\leq p_i<1$ for all $i>0$ and let $f_{\bp}(n):=\log(\prod_{i=1}^n(1-p_i))/\log(n)$.
If $\lim_{n\to\infty}f_{\bp}(n)=0$ then $M^n_{\bp}$ satisfies the 0-1 law for $L$.
\item The demand above on $f_{\bp}$ is the best possible. Formally for each $\epsilon>0$, there exists some $\bp$ with $0\leq p_i<1$ for all $i>0$
such that $|f_{\bp}(n)|<\epsilon$ but the 0-1 law fails for $M^n_{\bp}$.
\end{enumerate}
\end{thm}

Part (1) above gives a necessary condition on $\bp$ for the 0-1 law to hold in $M^n_{\bp}$, but the condition is not sufficient and a full characterization of $\bp$ seems to be harder. However we give below a complete characterization of $\bp$ in terms of the 0-1 law in $M^n_{\bq}$ for all $\bq$ "dominated by $\bp$", in the appropriate sense.  Alternatively one may ask which of the asymptotic properties of $M^n_{\bp}$ are kept under some operations on $\bp$. The notion of "domination" or the "operations" are taken from examples of the failure of the 0-1 law, and specifically the construction for part (2) above. Those are given in \cite{LuSh} by either adding zeros to a given sequence or decreasing some of the members of a given sequence. Formally define:

\begin{df}
For a sequence $\bp=(p_1,p_2,...)$:
\begin{enumerate}
    \item
    $Gen_1(\bp)$ is the set of all sequences $\bq=(q_1,q_2,...)$ obtained from $\bp$ by adding zeros to $\bp$. Formally $\bq\in Gen_1(\bp)$ iff for some increasing $f:\bN\to\bN$ we have for all $l>0$
    \begin{displaymath}
    q_l=
        \left\{ \begin{array}{ll}
        p_i&F(i)=l\\
        0&l\not\in Im(f).
    \end{array}\right.
    \end{displaymath}
    \item
    $Gen_2(\bp):=\{\bq=(q_1,q_2,...):l>0 \Rightarrow q_l\in [0,p_l]\}$.
    \item
    $Gen_3(\bp):=\{\bq=(q_1,q_2,...):l>0 \Rightarrow q_l\in\{0,p_l\}\}$.

\end{enumerate}
\end{df}

\begin{df}
Let $\bp=(p_1,p_2,...)$ be a sequence of probabilities and $\gL$ be some logic. For a sentence $\psi\in\gL$ denote by $Pr[M^n_{\bp}\models\psi]$ the probability that $\psi$ holds in $M^n_{\bp}$.
\begin{enumerate}
    \item
    We say that $M^n_{\bp}$ satisfies the 0-1 law for $\gL$, if for all $\psi\in\gL$ the limit $\lim_{n\to\infty}Pr[M^n_{\bp}\models\psi]$ exists and belongs to $\{0,1\}$.
    \item
    We say that $M^n_{\bp}$ satisfies the convergence law for $\gL$, if for all $\psi\in\gL$ the limit $\lim_{n\to\infty}Pr[M^n_{\bp}\models\psi]$ exists.
    \item
    We say that $M^n_{\bp}$ satisfies the weak convergence law for $\gL$, if for all $\psi\in\gL$, $\limsup_{n\to\infty}Pr[M^n_{\bp}\models\psi]-\liminf_{n\to\infty}Pr[M^n_{\bp}\models\psi]<1$.
    \item
    For $i\in\{1,2,3\}$ we say that $\bp$ $i$-hereditarily satisfies the 0-1 law for $\gL$, if for all $\bq\in Gen_i(\bp)$, $M^n_{\bq}$ satisfies the 0-1 law for $\gL$.
    \item
    Similarly to (4) for the convergence and weak convergence law.
\end{enumerate}
\end{df}

The main theorem of this paper is the following strengthening of theorem \ref{ThmPrevL}:
\begin{thm} \label{ThmMain}
Let $\bp=(p_1,p_2,...)$ be such that $0\leq p_i<1$ for all $i>0$, and $j\in\{1,2,3\}$. Then $\bp$ $j$-hereditarily satisfies the 0-1 law for $L$ iff $$(*)\qquad\qquad\lim_{n\to\infty}\log(\prod_{i=1}^n(1-p_i))/\log n=0.$$
Moreover we may replace above the "0-1 law" by the "convergence law" or "weak convergence law".
\end{thm}

Note that the 0-1 law implies the convergence law which in turn implies the weak convergence law. Hence it is enough to prove the "if" direction for the 0-1 law and the "only if" direction for the weak convergence law. Also note that the "if" direction is an immediate conclusion of Theorem \ref{ThmPrevL} (in the case $j=1$ it is stated in \cite{LuSh} as a corollary at the end of section 3). The case $j=1$ is proved in section 2, and the case $j\in\{2,3\}$ is proved in section 3. In section 4 we deal with the case $U^*(\bp):=\{i:p_i=1\}$ is not empty. We give an almost full analysis of the hereditary $0-1$ law in this case as well. The only case which is not fully characterized is the case $j=1$ and $|U^*(\bp)|=1$. We give some results regarding this case in section 5. The case $j=1$ and $|U^*(\bp)|=1$ and the case that the successor relation belongs to the dictionary, will be dealt with in \cite{Sh}. The following table summarizes the results in this article regarding the $j$-hereditary laws.

\begin{tabular}{|c||c|c|c|c|}
\hline
&\multicolumn{1}{c|}{$|U^*|=\infty$} & $2 \leq |U^*|< \infty$ & $|U^*|=1$ & \multicolumn{1}{c|}{$|U^*|=0$} \\
\hline
\hline

&  & The 0-1 law holds & See & \\
$j=1$ &  & $\Updownarrow$ & section & $\lim_{n\to\infty}\frac{\log(\prod_{i=1}^n(1-p_i))}{\log n}=0$\\
&The weak & $\{l:0<p_l<1\}=\emptyset$ & 5 & $\Updownarrow$\\
\cline{1-1}\cline{3-4}

&&\multicolumn{2}{c|}{The 0-1 law holds} & The 0-1 law holds\\
$j=2$ & convergence  & \multicolumn{2}{c|}{$\Updownarrow$} & $\Updownarrow$\\
& &\multicolumn{2}{c|}{$|\{l:p_l>0\}|\leq1$} &  The convergence law holds\\
\cline{1-1}\cline{3-4}

&  law fails & \multicolumn{2}{c|}{The 0-1 law holds} & $\Updownarrow$ \\
$j=3$ &&\multicolumn{2}{c|}{$\Updownarrow$} & The weak convergence law holds\\
&& \multicolumn{2}{c|}{$\{l:0<p_l<1\}=\emptyset$} &\\
\hline
\end{tabular}

\begin{conv}
Formally speaking Definition \ref{DefM} defines a probability on the
space of subsets of $G^n:=\{G:G\text{ is a graph with vertex set }[n]\}$. If $H$ is a
subset of $G^n$ we denote its probability by $Pr[M^n_{\bp}\in H]$. If
$\phi$ is a sentence in some logic we write
$Pr[M^n_{\bp}\models\phi]$ for the probability of
$\{G\in G^n:G\models\phi\}$. Similarly if $A_n$ is some property of
graphs on the set of vertexes $[n]$, then we write $Pr[A_n]$ or $Pr[A_n\text{ holds in }M^n_{\bp}]$ for the
probability of the set $\{G\in G^n:G\text{ has the property
}A_n\}$.
\end{conv}

\begin{nota}
\begin{enumerate}
\item $\bN$ is the set of natural numbers (including $0$).
\item $n,m,r,i,j$ and $k$ will denote natural numbers. $l$ will denote a member of $\bN^*$ (usually an index).
\item $p,q$ and similarly $p_l,q_l$ will denote probabilities i.e. reals in $[0,1]$.
\item $\epsilon,\zeta$ and $\delta$ will denote positive reals.
\item $L=\{\thicksim\}$ is the vocabulary of graphs i.e $\thicksim$ is a binary relation symbol. All $L$-structures are assumed to be graphs i.e. $\backsim$ is interpreted by a symmetric non-reflexive binary relation.
\item If $x\thicksim y$ holds in some graph $G$, we say that $\{x,y\}$ is an edge of $G$ or that $x$ and $y$ are "connected" or "neighbors" in $G$.
\end{enumerate}
\end{nota}

\section{Adding zeros}
In this section we prove theorem \ref{ThmMain} for $j=1$.
As the "if" direction is immediate from Theorem \ref{ThmPrevL} it remains to prove that if $(*)$ of \ref{ThmMain} fails then the 0-1 law for $L$ fails for some $\bq\in Gen_1(\bp)$. In fact we will show that it fails "badly" i.e. for some $\psi\in L$, $Pr[M^n_{\bq}\models\psi]$ approaches both $0$ and $1$ simultaneously. Formally:

\begin{df}
\begin{enumerate}
    \item
    Let $\psi$ be a sentence in some logic $\gL$, and $\bq=(q_1,q_2,...)$ be a series of probabilities. We say that $\psi$ holds infinitely often in $M^n_{\bar{q}}$ if $\limsup_{n\to\infty}Prob[M^n_{\bar{q}}\models\psi]=1$.
    \item
    We say that the 0-1 law for $\gL$ strongly fails in $M^n_{\bar{q}}$, if for some $\psi\in\gL$ both $\psi$ and $\lnot\psi$ hold infinitely often in $M^n_{\bar{q}}$.
\end{enumerate}
\end{df}

Obviously the 0-1 law strongly fails in some $M^n_{\bar{q}}$ iff $M^n_{\bar{q}}$ does not satisfy the weak semi 0-1 law. Hence in order to prove Theorem \ref{ThmMain} for $j=1$ it is enough if we prove:

\begin{lem} \label{LemSec1}
Let $\bp=(p_1,p_2,...)$ be such that $0\leq p_i<1$ for all $i>0$, and assume that $(*)$ of \ref{ThmMain} fails. Then for some $\bar{q}\in Gen_1(\bar{p})$ the 0-1 law for $L$ strongly fails in $M^n_{\bar{q}}$.
\end{lem}

In the remainder of this section we prove Lemma \ref{LemSec1}. We do so by inductively constructing $\bq$, as the limit of a series of finite sequences. Let us start with some basic definitions:
\begin{df} \label{DefP}
\begin{enumerate}
    \item
    Let $\gP$ be the set of all, finite or infinite, sequences of probabilities. Formally each $\bar{p}\in\gP$ has the form $\langle p_l:0<l<
    n_{\bp}\rangle$ where each $p_l\in[0,1]$ and $n_{\bp}$ is either $\omega$ (the first infinite ordinal) or a member of $\bN\setminus\{0,1\}$. Let $\gP^{inf}=\{\bp\in\gP:n_{\bp}=\omega\}$, and $\gP^{fin}:=\gP\setminus\gP^{inf}$.
    \item
    For $\bq\in\gP^{fin}$ and increasing $f:[n_{\bq}]\to\bN$, define $\bq^f\in\gP^{fin}$ by $n_{\bq^f}=f(n_{\bq})$, $(\bq^f)_l=q_i$ if $f(i)=l$ and $(\bq^f)_l=0$ if $l\not\in Im(f)$.
    \item
    For $\bp\in\gP^{inf}$ and $r>0$, let $Gen_1^r(\bp):=\{\bq\in\gP^{fin}:\text{ for some increasing } f:[r+1]\to\bN, (\bp|_{[r]})^f=\bq\}$.
    \item
    For $\bp,\bp'\in\gP$ denote $\bp\vartriangleleft\bp'$ if $n_{\bp}<n_{\bp'}$ and for each $l<n_{\bp}$, $p_l=p'_l$.
    \item
    If $\bp\in\gP^{fin}$ and $n>n_{\bp}$, we can still consider $M^n_{\bp}$ by putting $p_l=0$ for all $l\geq n_{\bp}$.
\end{enumerate}
\end{df}

\begin{obs} \label{Obs}
\begin{enumerate}
    \item
    Let $\langle \bp_i:i\in\bN\rangle$ be such that each $\bp_i\in\gP^{fin}$, and assume that $i<j\in\bN\Rightarrow\bp_i\vartriangleleft\bp_j$. Then
    $\bp=\cup_{i\in\bN}\bp_i$ (i.e. $p_l=(p_i)_l$ for some $\bp_i$ with $n_{\bp_i}>l$) is well defined and $\bp\in\gP^{inf}$.
    \item
    Assume further that $\langle r_i:i\in\bN\rangle$ is non-decreasing and unbounded, and that $\bp_i\in Gen_1^{r_i}(\bp')$ for some fixed $\bp'\in\gP^{inf}$, then $\cup_{i\in\bN}\bp_i\in Gen_1(\bp')$.
\end{enumerate}
\end{obs}

We would like our graphs $M^n_{\bq}$ to have a certain structure, namely that  the number of triangles in $M^n_{\bq}$ is $o(n)$ rather then say $o(n^3)$. we can impose this structure by making demands on $\bq$. This is made precise by the following:

\begin{df} \label{DefProper}
A sequence $\bq\in\gP$ is called proper (for $l^*$), if:
\begin{enumerate}
    \item $l^*$ and $2l^*$ are the first and second members of $\{0<l<n_{\bq}:q_l>0\}$.
    \item Let $l^{**}=3l^*+2$. If $l<n_{\bq}$, $l\not\in\{l^*,2l^*\}$ and $q_l>0$, then $l\equiv 1 \pmod l^{**}$.
\end{enumerate}

For $\bq,\bq'\in\gP$ we write $\bq\vartriangleleft^{prop}\bq'$ if $\bq\vartriangleleft\bq'$, and both $\bq$ and $\bq'$ are proper.
\end{df}

\begin{obs} \label{Obs1}
\begin{enumerate}
    \item If $\langle \bp_i:i\in\bN\rangle$ is such that each $\bp_i\in\gP$, and $i<j\in\bN\Rightarrow\bp_i\vartriangleleft^{prop}\bp_j$, then $\bp=\cup_{i\in\bN}\bp_i$ is proper.
    \item Assume that $\bq\in\gP$ is proper for $l^*$ and $n\in\bN$. Then the following event holds in $M^n_{\bq}$ with probability 1:
        \begin{itemize}
            \item[$(*)_{{\bq},l^*}$] If $m_1,m_2,m_3\in[n]$ and $\{m_1,m_2,m_3\}$ is a triangle in $M^n_{\bq}$, then $\{m_1,m_2,m_3\}=\{l,l+l^*,l+2l^*\}$ for some $l>0$.
        \end{itemize}
\end{enumerate}
\end{obs}

We can now define the sentence $\psi$ for which we have failure of the 0-1 law.
\begin{df} \label{DefPsi}
Let $k$ be an even natural number. Let $\psi_k$ be the $L$ sentence "saying":
There exists $x_0,x_1,...,x_k$ such that:
\begin{itemize}
    \item $(x_0,x_1,...,x_k)$ is without repetitions.
    \item For each even $0\leq i < k$, $\{x_i,x_{i+1},x_{i+2}\}$ is a triangle.
    \item The valency of $x_0$ and $x_k$ is 2.
    \item For each even $0<i<k$ the valency of $x_i$ is 4.
    \item For each odd $0<i<k$ the valency of $x_i$ is 2.
\end{itemize}
If the above holds (in a graph $G$) we say that $(x_0,x_1,...,x_k)$ is a chain of triangles (in $G$).
\end{df}

\begin{df}
Let $n\in\bN$, $k\in\bN$ be even and $l^*\in[n]$. For $1\leq m < n-k\cdot l^*$ a sequence $(m_0,m_1,...,m_k)$ is called a candidate of type $(n,l^*,k,m)$ if it is without repetitions, $m_0=m$ and for each even $0\leq i<k$, $\{m_i,m_{i+1},m_{i+2}\}=\{l,l+l^*,l+2l^*\}$ for some $l > 0$. Note that for given $(n,l^*,k,m)$, there are at most $4$ candidates of type $(n,l^*,k,m)$ (and at most $2$ if $k>2$).
\end{df}

\begin{claim} \label{ClaimProp}
Let $n\in\bN$, $k\in\bN$ be even, and $\bq\in\gP$ be proper for $l^*$. For $1\leq m < n-k\cdot l^*$
let $E^n_{\bq,m}$ be the following event (on the probability space $M^n_{\bq}$): "No candidate of of type $(n,l^*,k,m)$ is a chain of triangles."
Then $M^n_{\bq}$ satisfies with probability 1: $M^n_{\bq}\models\lnot\psi_k$ iff $M^n_{\bq}\models\bigwedge_{1\leq m < n-k\cdot l^*}E^n_{\bq,m}$
\end{claim}

\begin{proof}
The "only if" direction is immediate. For the "if" direction note that by \ref{Obs1}(2), with probability $1$, only a candidate can be a chain of triangles, and the claim follows immediately.
\end{proof}

The following claim shows that by adding enough zeros at the end of $\bq$ we can make sure that $\psi_k$ holds in $M^n_{\bq}$ with probability close to 1. Note that we do not make a "strong" use of the properness of $\bq$, i.e we do not use item (2) of Definition \ref{DefProper}.
\begin{claim} \label{ClaimPos}
Let $\bq\in\gP^{fin}$ be proper for $l^*$, $k\in\bN$ be even, and $\zeta>0$ be some rational. Then there exists $\bq'\in\gP^{fin}$ such that $\bq\vartriangleleft^{prop}\bq'$ and $Pr[M^{n_{\bq'}}_{\bq'}\models\psi_k]\geq1-\zeta$.
\end{claim}

\begin{proof}
For $n>n_{\bq}$ denote by $\bq^n$ the member of $\gP$ with $n_{\bq^n}=n$ and $(q^n)_l$ is $q_l$ if $l< n_{\bq}$ and $0$ otherwise. Note that $\bq\vartriangleleft^{prop}\bq^n$, hence if we show that for $n$ large enough we have $Pr[M^n_{\bq^n}\models\psi_k]\geq1-\zeta$ then we will be done by putting $\bq'=\bq^n$. Note that (recalling Definition \ref{DefP}(5)) $M^n_{\bq}=M^n_{\bq^n}$ so below we may confuse between them.  Now set $n^*=\max\{n_{\bq},k\cdot l^*\}$. For any $n>n^*$ and $1\leq m\leq n-n^*$ consider the sequence $s(m)=(m,m+l^*,m+2l^*,...,m+k\cdot l^*)$ (note that $s(m)$ is a candidate of type $(n,l^*,k,m)$). Denote by $E_m$ the event that $s(m)$ is a chain of triangles (in $M^n_{\bq}$). We then have:
$$Pr[M^n_{\bq}\models E_m]\geq (q_{l^*})^k\cdot(q_{2l^*})^{k/2}\cdot(\prod_{l=1}^{n_{\bq}-1}(1-p_l))^{2(k+1)}.$$
Denote the expression on the right by $p^*_{\bq}$ and note that it is positive and  depends only on $k$ and $\bq$ (but not on $n$). Now assume that $n>6\cdot n^*$ and that $1\leq m<m'\leq n-n^*$ are such that $m'-m > 2\cdot n^*$. Then the distance between the sequences $s(m)$ and $s(m')$ is larger than $n_{\bq}$ and hence the events $E_m$ and $E_{m'}$ are independent. We conclude that $Pr[M^n_{\bq}\not\models\psi_k] \leq (1-p^*_{\bq})^{n/(2\cdot n^*+1)}\rightarrow_{n\to\infty} 0$
and hence by choosing $n$ large enough we are done.
\end{proof}

The following claim shows that under our assumptions we can always find a long initial segment $\bq$ of some member of $Gen_1(\bp)$ such that $\psi_k$ holds in $M^n_{\bq}$ with probability close to 0. This is where we make use of our assumptions on $\bp$ and the properness of $\bq$.

\begin{claim} \label{ClaimNeg}
Let $\bp\in\gP^{inf}$, $\epsilon>0$ and assume that for an unbounded set of $n\in\bN$ we have $\prod_{l=1}^n(1-p_l)\leq n^{-\epsilon}$. Let $k\in\bN$ be even such that $k\cdot\epsilon > 2$. Let $\bq\in Gen_1^r(\bp)$ be proper for $l^*$, and $\zeta>0$ be some rational. Then there exists $r'>r$ and $\bq'\in Gen_1^{r'}(\bp)$ such that $\bq\vartriangleleft^{prop}\bq'$ and $Pr[M^{n_{\bq'}}_{\bq'}\models\lnot\psi_k]\geq1-\zeta$.
\end{claim}

\begin{proof}
First recalling Definition \ref{DefProper} let $l^{**}=3l^*+2$, and for $l\geq n_{\bq}$ define $r(l):=\lceil (l-n_{\bq}+1)/l^{**}\rceil$. Now
for each $n>n_{\bq}+l^{**}$ denote by $\bq_n$ the member of $\gP$ defined by:
\begin{displaymath}
(q_n)_l=
\left\{ \begin{array}{ll}
      q_l&0<l< n_{\bq}\\
      0&n_{\bq}\leq l<n \textrm{ and } l\not\equiv 1 \mod l^{**}\\
      p_{r+r(l)} & n_{\bq}\leq l<n \textrm{ and } l\equiv 1 \mod l^{**}.
\end{array}\right.
\end{displaymath}
Note that $n_{\bq_n}=n$, $\bq_n\in Gen_1^{r'}(\bp)$ where $r'=r+r(n-1)>r$ and $\bq\vartriangleleft^{prop}\bq_n$. Hence if we show that for some $n$ large enough we have $Pr[M^n_{\bq_n}\models\lnot\psi_k]\geq1-\zeta$ then we will be done by putting $\bq'=\bq_n$. As before let $n^*:=\max\{kl^*,n_{\bq}+l^*\}$. Now fix some $n>n^*$ and for $1\leq m < n-k\cdot l^*$ let $s(m)$ be some candidate of type $(n,l^*,k,m)$. Denote by $E=E(s(m))$ the event that $s(m)$ is a chain of triangles in $M^n_{\bq_n}$. We then have:
$$Pr[M^n_{\bq_n}\models E] \leq (q_{l^*})^k\cdot(q_{2l^*})^{k/2} \cdot (\prod_{n^*+1}^{\lfloor(n-n^*)/2\rfloor}(1-(q_i)_l))^{k}.$$
Now denote:
$$p^*_{\bq}:= (q_{l^*})^k\cdot(q_{2l^*})^{k/2} \cdot (\prod_{l=1}^{n^*}(1-(q_i)_l))^{-k}$$
and note that it is positive and does not depend on $n$. Together we get:
$$ Pr[M^n_{\bq_n}\models E] \leq p^* \cdot (\prod_{l=1}^{\lfloor(n-n^*)/2\rfloor}(1-(q_i)_l))^{k}\leq p^*_{\bq} \cdot (\prod_{l=1}^{\lfloor(n-n^*)/(2l^{**})\rfloor}(1-p_l))^{k}.$$
For each $1\leq m < n-k\cdot l^*$ the number of candidates of type $(n,l^*,k,m)$ is at most $4$, hence the total number of candidates is no more then $4n$. We get that the expected number (in the probability space $M^n_{\bq_n}$) of candidates which are a chain of triangles is at most $p^*_{\bq}\cdot (\prod_{l=1}^{\lfloor(n-n^*)/(2l^{**})\rfloor}(1-p_l))^{k}\cdot 4n$. Let $E^*$ be the following event: "No candidate is a chain of triangles". Then using Claim \ref{ClaimProp} and Markov's inequality we get: $$Pr[M^n_{\bq}\models\psi_k] = Pr[M^n_{\bq}\not\models E^*] \leq p^*_{\bq}\cdot (\prod_{l=1}^{\lfloor(n-n^*)/(2l^{**})\rfloor}(1-p_l))^{k}\cdot 4n.$$ Finally by our assumptions, for an unbounded $n$ we have $\prod_{l=1}^{\lfloor(n-n^*)/(2l^{**})\rfloor}(1-p_l)\leq (\lfloor(n-n^*)/(2l^{**})\rfloor)^{-\epsilon}$, and note that for $n$ large enough we have $(\lfloor(n-n^*)/(2l^{**})\rfloor)^{-\epsilon}\leq n^{-\epsilon/2}$. Hence for unbounded $n\in\bN$ we have $Pr[M^n_{\bq}\models\psi_k]\leq p^*_{\bq}\cdot 4\cdot n^{1-\epsilon\cdot k/2}$, and as $\epsilon\cdot k>2$ this tends to $0$ as $n$ tends to $\infty$, so we are done.
\end{proof}

We are now ready to prove Lemma \ref{LemSec1}. First as $(*)$ of \ref{ThmMain} does not hold we have some $\epsilon>0$ such that for an unbounded set of $n\in\bN$, we have $\prod_{l=1}^n(1-p_l)\leq n^{-\epsilon}$. Let $k\in\bN$ be even such that $k\cdot\epsilon>2$. Now
for each $i\in\bN$ we will construct a pair $(\bq_i,r_i)$ such that the following holds:
\begin{enumerate}
    \item
    For $i\in\bN$, $\bq_i\in Gen_1^{r_i}(\bp)$ and put $n_i:=n_{\bq_i}$.
    \item
    For $i\in\bN$, $\bq_i \vartriangleleft^{prop} \bq_{i+1}$.
    \item
    For each odd $i>0$, $Pr[M^{n_i}_{\bq_i}\models\psi_k]\geq1-\frac{1}{i}$ and $r_i=r_{i-1}$.
    \item
    For each even $i>0$, $Pr[M^{n_i}_{\bq_i}\models\lnot\psi_k]\geq1-\frac{1}{i}$ and $r_i>r_{i-1}$.
\end{enumerate}
Clearly if we construct such $\langle(\bq_i,r_i):i\in\bN\rangle$ then by taking $\bq=\cup_{i\in\bN}\bq_i$ (recall observation \ref{Obs}), we have $\bq\in Gen_1(\bp)$ and both $\psi_k$ and $\lnot\psi_k$ holds infinitely often in $M^n_{\bar{q}}$, thus finishing the proof. We turn to the construction of $\langle(\bq_i,r_i):i\in\bN\rangle$, and naturally we use induction on $i\in\bN$.

\textbf{Case 1:} $i=0$. Let $l_1<l_2$ be the first and second indexes such that $p_{l_i}>0$. Put $r_0:=l_2$. If $l_2\leq 2l_1$ define $\bq_0$ by:
\begin{displaymath}
(q_0)_l=
\left\{ \begin{array}{ll}
      p_l&l\leq l_1\\
      0&l_1\leq l\leq 2l_1\\
      p_{l_2} & l=2l_1.
\end{array}\right.
\end{displaymath}
Otherwise if $l_2>2l_1$ define $\bq_0$ by:
\begin{displaymath}
(q_0)_l=
\left\{ \begin{array}{ll}
      0&l< \lceil l_2/2\rceil\\
      p_{l_1}&l=\lceil l_2/2\rceil\\
      0&\lceil l_2/2\rceil<l< 2\lceil l_2/2\rceil\\
      p_{l_2}&l=2\lceil l_2/2\rceil.
\end{array}\right.
\end{displaymath}
clearly $\bq_0\in Gen_1^{r_0}(\bp)$ as desired, and note that $\bq_0$ is proper (for either $l_1$ or $\lceil l_2/2\rceil$).

\textbf{Case 2:} $i>0$ is odd. First set $r_i=r_{i-1}$. Next we use Claim \ref{ClaimPos} where we set: $\bq_{i-1}$ for $\bq$, $\frac{1}{i}$ for $\zeta$ and $\bq_i$ is the one promised by the claim. Note that indeed $\bq_{i-1} \vartriangleleft^{prop} \bq_{i}$, $\bq_i\in gen^{r_i}(\bp)$ and $Pr[M^{n_i}_{\bq_i}\models\psi_k]\geq1-\frac{1}{i}$.

\textbf{Case 3:} $i>0$ is even. We use Claim \ref{ClaimNeg} where we set: $\bq_{i-1}$ for $\bq$, $\frac{1}{i}$ for $\zeta$ and $(r_i,\bq_i)$ are $(r',\bq')$ promised by the claim. Note that indeed $\bq_{i-1} \vartriangleleft^{prop} \bq_{i}$, $\bq_i\in Gen_1^{r_i}(\bp)$ and $Pr[M^{n_i}_{\bq_i}\models\psi_k]\geq1-\frac{1}{i}$.
This completes the proof of Lemma \ref{LemSec1}.

\section{Decreasing coordinates}
In this section we prove Theorem \ref{ThmMain} for $j\in\{2,3\}$.
As before, the "if" direction is an immediate conclusion of Theorem \ref{ThmPrevL}. Moreover as $Gen_3(\bp)\subseteq Gen_2(\bp)$ it remains to prove that if $(*)$ of \ref{ThmMain} fails then the 0-1 strongly fails for some $\bq\in Gen_3(\bp)$. We divide the proof into two cases according to the behavior of $\sum_{l=1}^np_i$, which is an approximation of the expected number of neighbors of a given node in $M^n_{\bp}$. Define:
$$(**)\qquad\qquad\lim_{n\to\infty}\log(\sum_{i=1}^n p_i)/\log n=0.$$
Assume that $(**)$ above fails. Then for some $\epsilon>0$, the set $\{n\in\bN:\sum_{i=1}^n p_i\geq n^{\epsilon}\}$ is unbounded, hence we finish by Lemma \ref{LemA}. On the other hand if $(**)$ holds then $\sum_{i=1}^n p_i$ increases slower then any positive power of $n$, formally for all $\delta>0$ for some $n_{\delta}\in\bN$ we have $n>n_{\delta}$ implies $\sum_{i=1}^n p_i\leq n^{\delta}$. As we assume that $(*)$ of Theorem \ref{ThmMain} fails we have for some $\epsilon>0$ the set $\{n\in\bN:\prod_{i=1}^n (1-p_i)\leq n^{-\epsilon}\}$ is unbounded. Together (with $-\epsilon/6$ as $\delta$) we have that the assumptions of Lemma \ref{LemB} hold, hence we finish the proof.

\begin{lem} \label{LemA}
Let $\bp\in\gP^{inf}$ be such that $p_l<1$ for $l>0$. Assume that for some $\epsilon >0$ we have
for an unbounded set of $n\in\bN$: $\sum_{l\leq n}p_l\geq n^{\epsilon}$.
Then for some $\bq\in Gen_3(\bp)$ and $\psi = \psi_{isolated}:=\exists x \forall y \lnot x\thicksim y$,
both $\psi$ and $\lnot\psi$ holds infinitely often in $M^n_{\bq}$.
\end{lem}

\begin{proof}
We construct a series, $(\bq_1,\bq_2,...)$ such that for $i>0$: $\bq_i\in\gP^{fin}$, $\bq_i\vartriangleleft\bq_{i+1}$ and $\cup_{i>0}\bq_i\in Gen_3(\bp)$. For $i\geq1$ denote $n_i:=n_{\bq_i}$. We will show that:
\begin{itemize}
    \item[$*_{even}$] For even $i>1$: $Pr[M^{n_i}_{\bq_i}\models\psi]\geq1-\frac{1}{i}$.
    \item[$*_{odd}$] For odd $i>1$: $Pr[M^{n_i}_{\bq_i}\models\lnot\psi]\geq1-\frac{1}{i}$.
\end{itemize}
Taking $\bq=\cup_{i>0}\bq_i$ will then complete the proof. We construct $\bq_i$ by induction on $i>0$:

\textbf{Case 1} $i=1$: Let $n_1=2$ and $(q_1)_1=p_1$.

\textbf{Case 2} even $i>1$: As $(\bq_{i-1},n_{i-1})$ are given, let us define $\bq_i$ were $n_i>n_{i-1}$ is to be determined later:
$(q_i)_l=(q_{i-1})_l$ for $l< n_{i-1}$ and $(q_i)_l=0$ for $n_{i-1}\leq l< n_i$. For $x\in [n_i]$ let $E_x$ be the event: "$x$ is an isolated point".
Denote $p':=(\prod_{0<l< n_{i-1}}(1-(q_{i-1})_l)^2$ and note that $p'>0$ and does not depend on $n_i$. Now for $x\in [n_i]$, $Pr[M^{n_i}_{\bq_i}\models E_x]\geq p'$, furthermore if $x,x'\in [n_i]$ and $|x-x'|>n_{i-1}$ then $E_x$ and $E_{x'}$ are independent in $M^{n_i}_{\bq_i}$. We conclude that $Pr[M^{n_i}_{\bq_i}\models\lnot\psi]\leq(1-p)^{\lfloor n_i/(n_{i-1}+1)\rfloor}$ which approaches $0$ as $n_i\to\infty$. So by choosing $n_i$ large enough we have $*_{even}$.

\textbf{Case 3} odd $i>1$: As in case 2 let us define $\bq_i$ were $n_i>n_{i-1}$ is to be determined later:
$(q_i)_l=(q_{i-1})_l$ for $l< n_{i-1}$ and $(q_i)_l=p_l$ for $n_{i-1}\leq l< n_i$. Let $n'=\max\{n<n_i/2:n=2^m\text{ for some }m\in\bN\}$, so $n_i/4\leq n' <n_i/2$. Denote $a=\sum_{0<l\leq n'} (q_i)_l$ and $a'=\sum_{0<l\leq \lfloor n/4 \rfloor} (q_i)_l$. Again let $E_x$ be the event: "$x$ is isolated". Now as $n'<n_i/2$, $Pr[M^{n_i}_{\bq_i}\models E_x]\leq\prod_{0<l\leq n'}(1-(q_i)_l)$. By a repeated use of: $(1-x)(1-y)\leq(1-\frac{x+y}{2})^2$ we get $Pr[M^{n_i}_{\bq_i}\models E_x]\leq(1-\frac{a}{n'})^{n'}$ which for $n'$ large enough is smaller then $2\cdot e^{-a}$, and as $a'\leq a$, we get $Pr[M^{n_i}_{\bq_i}\models E_x]\leq2\cdot e^{-a'}$. By the definition of $a'$ and $\bq_i$ we have $a'=\sum_{l=1}^{\lfloor n_1/4 \rfloor}p_l-\sum_{l<n_{i-1}}(p_l-(q_{i-1})_l)$. By our assumption for an unbounded set of $n_i\in\bN$ we have $a'\geq (\lfloor n_i/4 \rfloor)^{\epsilon}-\sum_{l<n_{i-1}}(p_l-(q_{i-1})_l)$. But as the sum on the right is independent of $n_i$ we have (again for $n_i$ large enough): $a'\geq(n_i/5)^{\epsilon}$. Consider the expected number of isolated points in the probability space $M^{n_i}_{\bq_i}$, denote this number by $X(n_i)$. By all the above we have:
$$X(n_i)\leq n_i\cdot 2\cdot e^{-a}\leq n_i\cdot 2\cdot e^{-a'}\leq 2n_i\cdot e^{-(n_i/5)^{\epsilon}}.$$
The last expression approaches $0$ as $n_i\to\infty$. So by choosing $n_i$ large enough (while keeping $a'\geq(n_i/5)^{\epsilon}$ we have $*_{odd}$.

Finally notice that indeed $\cup_{i>0}\bq_i\in Gen_3(\bp)$, as the only change we made in the inductive process is decreasing $p_l$ to $0$ for $n_{i-1}<l\leq n_i$ and $i$ is even.
\end{proof}

\begin{lem} \label{LemB}
Let $\bp\in\gP^{inf}$ be such that $p_l<1$ for $l>0$. Assume that for some $\epsilon >0$ we have
for an unbounded set of $n\in\bN$:
\begin{itemize}
    \item[$(\alpha)$] $\sum_{l\leq n}p_l\leq n^{\epsilon/6}$.
    \item[$(\beta)$] $\prod_{l\leq n}(1-p_l)\leq n^{-\epsilon}$.
\end{itemize}
Let $k=\lceil\frac{6}{\epsilon}\rceil+1$ and $\psi=\psi_k$ be the sentence "saying" there exists a connected component which is a path of length $k$, formally:
$$\psi_k:=\exists x_1...\exists x_k\bigwedge_{1\leq i\neq j \leq k} x_i\neq x_j\wedge \bigwedge_{1\leq i<k} x_i \thicksim x_{i+1}\wedge\forall y(\bigwedge_{1\leq i \leq k} x_i\neq y)\to(\bigwedge_{1\leq i \leq k}\lnot x_i\thicksim y).$$
Then for some $\bq\in Gen_3(\bp)$,
both $\psi$ and $\lnot\psi$ holds infinitely often in $M^n_{\bq}$.
\end{lem}

\begin{proof}
The proof follows the same line as the proof of \ref{LemA}. We construct an increasing series, $(\bq_1,\bq_2,...)$, and demand $*_{even}$ and $*_{odd}$ as in \ref{LemA}. Taking $\bq=\cup_{i>0}\bq_i$ will then complete the proof. We construct $\bq_i$ by induction on $i>0$:

\textbf{Case 1} $i=1$: Let $l(*):=\min\{l>0:p_l>0\}$ and define $n_1=l(*)+1$ and $(q_1)_l=p_l$ for $l<n_1$.

\textbf{Case 2} even $i>1$: As before, for $n_i>n_{i-1}$ define:
$(q_i)_l=(q_{i-1})_l$ for $l< n_{i-1}$ and $(q_i)_l=0$ for $n_{i-1}\leq l< n_i$. For $1\leq x< n_i-k\cdot l(*)$ let $E^x$ be the event: "$(x,x+l(*),...,x+l(*)(k-1))$ exemplifies $\psi$." Formally $E^x$ holds in $M^{n_i}_{\bq_i}$ iff $\{(x,x+l(*),...,x+l(*)(k-1))\}$ is isolated and for $0\leq j<k-1$, $\{x+jl(*),x+(j+1)l(*)\}$ is an edge of $M^{n_i}_{\bq_i}$. The remainder of this case is similar to case 2 of Lemma \ref{LemA} so we will not go into details.
Note that $Pr[M^{n_i}_{\bq_i}\models E^x]>0$ and does not depend on $n_i$, and if $|x-x'|$ is large enough (again not depending on $n_i$) then $E^x$ and $E^{x'}$ are independent in $M^{n_i}_{\bq_i}$. We conclude that by choosing $n_i$ large enough we have $*_{even}$.

\textbf{Case 3} odd $i>1$: In this case we make use of the fact that almost always, no $x\in[n]$ have to many neighbors. Formally:

\begin{claim}
Let $\bq\in\gP^{inf}$ be such that $q_l<1$ for $l>0$. Let $\delta>0$ and assume that for an unbounded set of $n\in\bN$ we have, $\sum_{l=1}^nq_l\leq n^{\delta}$. Let $E^n_{\delta}$ be the event: "No $x\in[n]$ have more than $8n^{2\delta}$ neighbors". Then we have:
$$\limsup_{n\to\infty}Pr[E^n_{\delta}\text{ holds in }M^n_{\bq}]=1.$$
\end{claim}

\begin{proof}
First note that the size of the set $\{l>0:q_l>n^{-\delta}\}$ is at most $n^{2\delta}$. Hence by ignoring at most $2n^{2\delta}$ neighbors of each $x\in[n]$, and changing the number of neighbors in the definition of $E^n_{\delta}$ to $6n^{2\delta}$ we may assume that for all $l>0$, $q_l\leq n^{-\delta}$. The idea is that the number of neighbors of each $x\in[n]$ can be approximated (or in our case only bounded from above) by a Poisson random variable with parameter close to $\sum_{i=l}^nq_l$. Formally, for each $l>0$ let $B_l$ be a Bernoulli random variable with $Pr[B_l=1]=q_l$. For $n\in\bN$ let $X^n$ be the random variable defined by $X^n:=\sum_{l=1}^nB_l$.  For $l>0$ let $Po_l$ be a Poisson random variable with parameter $\lambda_l:=-\log(1-q_l)$ that is for $i=0,1,2,...$ $Pr[Po_l=i]=e^{-\lambda_l}\frac{(\lambda_l)^i}{i!}$. Note that $Pr[B_l=0]=Pr[Po_l=0]$. Now define $Po^n:=\sum_{i=1}^nPo_l$. By the last sentence we have $Po^n\geq_{st}X^n$ ($Po^n$ is stochastically larger than $X^n$) that is, for $i=0,1,2,...$ $Pr[Po^n\geq i]\geq Pr[X^n\geq i]$. Now $Po^n$ (as the sum of Poisson random variables) is a Poisson random variable with parameter $\lambda^n:=\sum_{l=1}^n\lambda_l$. Let $n\in\bN$ be such that $\sum_{l=1}^nq_l\leq n^{\delta}$, and define $n'=n'(n):=min\{n'\geq n:n'=2^m\text{ for some } m\in\bN\}$, so $n\leq n'<2n$. For $0<l\leq n'$ let $q'_l$ be $q_l$ if $l\leq n$ and $0$ otherwise, so we have: $\prod_{l=1}^n 1-q_l=\prod_{l=1}^{n'} 1-q'_l$ and $\sum_{l=1}^{n}q_l=\sum_{l=1}^{n'}q'_l$. Note that if $0\leq p,q\leq 1/4$ then $(1-p)(1-q)\geq (1-\frac{p+q}{2})^2\cdot\frac{1}{2}$. By a repeated use of the last inequality we get that $\prod_{i=l}^{n'}(1-q'_l)\geq(1-\frac{\sum_{i=l}^{n'}q'_l}{n'})^{n'}\cdot\frac{1}{n'}$. We can now evaluate $\lambda^n$:
\begin{eqnarray*}
\lambda^n&=&\sum_{l=1}^{n}\lambda_l=\sum_{l=1}^{n}-\log(1-q_l)=-\log(\prod_{l=1}^{n}(1-q_l))=-\log(\prod_{l=1}^{n'}(1-q'_l))\\
&\leq&-\log[(1-\frac{\sum_{l=1}^{n'}q'_l}{n'})^{n'}\cdot\frac{1}{n'}]=-\log[(1-\frac{\sum_{l=1}^{n}q_l}{n'})^{n'}\cdot\frac{1}{n'}]\\
&\approx&-\log[e^{-\sum_{l=1}^{n}q_l}\cdot\frac{1}{n'}]\leq-\log[e^{-n^{\delta}}\cdot\frac{1}{2n}]\leq-\log[e^{-n^{2\delta}}]=n^{2\delta}.
\end{eqnarray*}
Hence by choosing $n\in\bN$ large enough while keeping $\sum_{l=1}^nq_l\leq n^{\delta}$ (which is possible by our assumption) we have $\lambda^n\leq n^{2\delta}$. We now use the Chernoff bound for Poisson random variable: If $Po$ is a Poisson random variable with parameter $\lambda$ and $i>0$ we have $Pr[Po\geq i]\leq e^{\lambda(i/\lambda-1)}\cdot(\frac{\lambda}{i})^i$. Applying this bound to $Po^n$ (for $n$ as above) we get:
$$Pr[Po^n\geq3n^{2\delta}]\leq e^{\lambda^n(3n^{2\delta}/\lambda^n-1)}\cdot(\frac{\lambda^n}{3n^{2\delta}})^{3n^{2\delta}}\leq e^{3n^{2\delta}}\cdot(\frac{\lambda^n}{3n^{2\delta}})^{3n^{2\delta}}\leq(\frac{e}{3})^{3n^{2\delta}}.$$
Now for $x\in[n]$ let $X^n_x$ be the number of neighbors of $x$ in $M^n_{\bq}$ (so $X^n_x$ is a random variable on the probability space $M^n_{\bq}$). By the definition of $M^n_{\bq}$ we have $X^n_x\leq_{st}2\cdot X^n\leq_{st}2\cdot Po^n$. So for unbounded $n\in\bN$ we have for all $x\in[n]$, $Pr[X^n_x\geq 6n^{2\delta}]\leq (\frac{e}{3})^{3n^{2\delta}}$. Hence by the Markov inequality for unbounded $n\in\bN$ we have, $$Pr[E^n\text{ does not hold in }M^n_{\bq}]=Pr[\text{for some } x\in[n], X^n_x\geq 3n^{2\delta}]\leq n\cdot(\frac{e}{3})^{6n^{2\delta}}.$$ But the last expression approaches $0$ as $n$ approaches $\infty$, Hence we are done proving the claim.
\end{proof}

We return to \textbf{Case 3} of the proof of \ref{LemB}, and it remains to construct $\bq_i$. As before for $n_i>n_{i-1}$ define:
$(q_i)_l=(q_{i-1})_l$ for $l< n_{i-1}$ and $(q_i)_l=p_l$ for $n_{i-1}\leq l< n_i$. By the claim above and $(\alpha)$ is our assumptions, for $n_i$ large enough we have
$Pr[E^{n_i}_{\epsilon/6}\text{ holds in }M^{n_i}_{\bq_i}]\geq 1/2i$, so assume in the rest of the proof that $n_i$ is indeed large enough, and assume that $E^{n_i}_{\epsilon/6}$ holds in $M^{n_i}_{\bq_i}$, and all the probabilities on the space $M^{n_i}_{\bq_i}$ will be conditioned to $E^{n_i}_{\epsilon/6}$ (even if not explicitly said so). A $k$-tuple $\bar{x}=(x_1,...,x_k)$ of members of $[n_i]$ is called a $k$-path (in $M^{n_i}_{\bq_i}$) if it is without repetitions and for $0<j<k$ we have $M^{n_i}_{\bq_i}\models x_j\thicksim x_{j+1}$. A $k$-path is isolated if in addition no member of $\{x_1,...,x_k\}$ is connected to a member of $[n_i]\setminus\{x_1,...,x_k\}$. Now (recall we assume $E^{n_i}_{\epsilon/6}$) with probability $1$: the number of $k$-paths in $M^{n_i}_{\bq_i}$ is at most $8^k\cdot n^{1+k\epsilon/3}$. For each $(x_1,...,x_k)$ without repetitions we have:
$$ Pr[(x_1,...,x_k)\text{ is isolated in }M^{n_i}_{\bq_i}]=\prod_{j=1}^k\prod_{y\neq x_j}(1-(q_i)_{|x_j-y|})\leq(\prod_{l=1}^{\lfloor n_i/2\rfloor}(1-(q_i)_l))^k.$$ By assumption $(\beta)$ we have for unbounded set of $n_i\in\bN$:
$$\prod_{l=1}^{\lfloor n_i/2\rfloor}(1-(q_i)_l)\leq\prod_{l=n_i-1}^{\lfloor n_i/2\rfloor}(1-p_l)\leq\prod_{l<n_i}(1-q_l)\cdot(\lfloor n_i/2\rfloor)^{-\epsilon}\leq (n_i)^{-\epsilon/2}.$$ Together letting $Y(n_i)$ be the expected number of isolated $k$ tuples in $M^{n_i}_{\bq_i}$ we have:
$$Y(n_i)\leq 8^k\cdot (n_i)^{1+k\epsilon/3}\cdot (n_i)^{-k\epsilon/2}=8^k\cdot (n_i)^{1-k\epsilon/6}\rightarrow_{n_i\to\infty}0.$$
So by choosing $n_i$ large enough and using Markov's inequality, we have $*_{odd}$, and we are done.
\end{proof}

\section{Allowing some probabilities to equal $1$}
In this section we analyze the hereditary 0-1 law for $\bp$ where some of the $p_i$-s may equal $1$. For $\bp\in\gP^{inf}$ let $U^*(\bp):=\{l>0:p_l=1\}$. The situation $U^*(\bp)\neq\emptyset$ was discussed briefly in the end of section 4 of \cite{LuSh}, an example was given there of some $\bp$ consisting of only ones and zeros with $|U^*(\bp)|=\infty$ such that the 0-1 law fails for $M^n_{\bp}$. We follow the lines of that example and prove that if $|U^*(\bp)|=\infty$ and $j\in\{1,2,3\}$, then the $j$-hereditary 0-1 law for $L$ fails for $\bp$. This is done in \ref{ThmInf}.
The case $0<|U^*(\bp)|<\infty$ is also studied and a full characterization of the $j$-hereditary 0-1 law for $L$ is given in \ref{ConcFin} for $j\in\{2,3\}$, and for $j=1$, $1<|U^*(\bp)|$. The case $j=1$ and $1=|U^*(\bp)|$ is discussed in section 5.

\begin{thm} \label{ThmInf}
Let $\bp\in\gP^{inf}$ be such that $U^*(\bp)$ is infinite, and $j$ be in $\{1,2,3\}$. Then
$M^n_{\bp}$ does not satisfy the $j$-hereditary weak convergence law for $L$.
\end{thm}

\begin{proof}
We start with the case $j=1$. The idea here is similar to that of section 2. We show that some $\bq\in Gen_1(\bp)$ has a structure (similar to the "proper" structure defined in \ref{DefProper}) that allows us to identify the sections "close" to $1$ or $n$ in $M^n_{\bq}$. It is then easy to see that if $\bq$ has infinitely many ones and infinitely many "long" sections of consecutive zeros, then the sentence saying: "there exists an edge connecting vertexes close to the the edges", will exemplify the failure of the 0-1 law for $M^n_{\bq}$. This is formulated below. Consider the following demands on $\bq\in\gP^{inf}$:
\begin{enumerate}
    \item
    Let $l^*<l^{**}$ be the first two members of $U^*(\bq)$, then $l^*$ is odd and $l^{**}=2\cdot l^*$.
    \item
    If $l_1,l_2,l_3$ all belong to $\{l>0:q_l>0\}$ and $l_1+l_2=l_3$ then $l_1=l_2=l^*$.
    \item
    The set $\{n\in\bN:n-2l^*<l<n\Rightarrow q_{l}=0\}$ is infinite.
    \item
    The set $U^*(\bq)$ is infinite.
\end{enumerate}
We first claim that some $\bq\in Gen_1(\bp)$ satisfies the demands (1)-(4) above. This is straight forward. We inductively add enough zeros before each nonzero member of $\bp$ guaranteing that it is larger than the sum of any two (not necessarily different) nonzero members preceding it. We continue until we reach $l^*$, then by adding zeros either before $l^*$ or before $l^{**}$ we can guarantee that $l^*$ is odd and that $l^{**}=2\cdot l^*$, and hence (1) holds. We then continue the same process from $l^{**}$, adding at least $2l^*$ zero's at each step. This guaranties (2) and (3). (4) follows immediately form our assumption that $U^*(\bp)$ is infinite. Assume that $\bq$ satisfies (1)-(4) and $n\in\bN$. With probability $1$ we have: $$\{x,y,z\}\text{ is a triangle in }M^n_{\bq}\text{ iff }\{x,y,z\}=\{l,l+l^*,l+l^{**}\}\text{ for some }0< l\leq n.$$
To see this use (1) for the "if" direction and (2) for the "only if" direction. We conclude that letting $\psi_{ext}(x)$ be the $L$ sentence saying that $x$ belongs to exactly one triangle, for each $n\in \bN$ and $m\in[n]$ with probability $1$ we have:
$$M^n_{\bq}\models\psi_{ext}[m]\text{ iff }m\in[1,l^*]\cup(n-l^*,n].$$
We are now ready to prove the failure of the weak convergence law in $M^n_{\bq}$, but in the first stage let us only show the failure of the convergence law. This will be useful for other cases (see Remark \ref{RemNew} below). Define $$\psi:=(\exists x \exists y)\psi_{ext}(x)\wedge\psi_{ext}(y)\wedge x\thicksim y.$$
Recall that $l^*$ is the \emph{first} member of $U^*(\bp)$, hence for some $p>0$ (not depending on $n$) for any $x,y\in[1,l^*]$ we have $Pr[M^n_{\bq}\models\lnot x\thicksim y]\geq p$ and similarly for any  $x,y\in(n-l^*,n]$. We conclude that:
$$Pr[(\exists x\exists y) (x,y\in[1,l^*] \text{ or } x,y\in(n-l^*,n])\text{ and } x\thicksim y] \leq 1-p^{2 {{l^*}\choose{2}}}<1.$$
By all the above, for each $l$ such that $q_l=1$ we have $Pr[M^{l+1}_{\bq}\models\psi]=1$, as the pair $(1,l+1)$ exemplifies $\psi$ in $M^{l+1}_{\bq}$ with probability $1$. On the other hand if $n$ is such that $n-2l^*<l<n \Rightarrow q_{l}=0$ then $Pr[M^{n}_{\bq}\models\psi]\leq 1-p^{2 {{l^*}\choose{2}}}$. Hence by (3) and (4) above, $\psi$ exemplifies the failure of the convergence law for $M^n_{\bq}$ as required.

We return to the proof of the failure of the weak convergence law. Define:
\begin{eqnarray*}
\psi'&=&\exists x_0...\exists x_{2l^*-1}[\bigwedge_{0\leq i<i'<2l^*}x_i\neq x_{i'}\wedge \forall y((\bigwedge_{0\leq i<2l^*}y\neq x_i)\to\lnot\psi_{ext}(y))\\&&\wedge\bigwedge_{0\leq i<2l^*}\psi_{ext}(x_i)\wedge\bigwedge_{0\leq i<l^*}x_{2i}\thicksim x_{2i+1}].
\end{eqnarray*}
We will show that both $\psi'$ and $\lnot\psi'$ holds infinitely often in $M^n_{\bq}$. First let $n\in\bN$ be such that $q_{n-l^*}=1$. Then by choosing for each $0\leq i<l^*$, $x_{2i}:=i+1$ and $x_{2i+1}:=n-l^*+1+i$, we will get that the sequence $(x_0,...,x_{2l^*-1})$ exemplifies $\psi'$ in $M^n_{\bq}$ (with probability 1). As by assumption (4) above the set $\{n\in\bN:q_{n-l^*}=1\}$ is unbounded we have $\limsup_{n\to\infty}[M^n_{\bq}\models\psi']=1$. For the other direction let $n\in\bN$ be such that for each $n-2l^*<l<n$, $q_l=0$. Then $M^n_{\bq}$ satisfies (again with probability 1) for each $x,y\in[1,l^*]\cup(n-l^*,n]$ such that $x\thicksim y$: $x\in[1,l^*]$ iff $y\in[1,l^*]$. Now assume that $(x_0,...,x_{2l^*-1})$ exemplifies $\psi'$ in $M^n_{\bq}$. Then for each $0\leq i<l^*$, $x_{2i}\in[1,l^*]$ iff $x_{2i+1}\in[1,l^*]$. We conclude that the set $[1,l^*]$ is of even size, thus contradicting (1). So we have $Pr[M^n_{\bq}\models\psi']=0$. But by assumption (3) above the set of natural numbers, $n$, for which we have $n-2l^*<l<n$ implies $q_l=0$ is unbounded, and hence we have $\limsup_{n\to\infty}[M^n_{\bq}\models\lnot\psi']=1$ as desired.

We turn to the proof of the case $j\in\{2,3\}$, and as $Gen_3(\bp)\subseteq Gen_2(\bp)$ it is enough to prove that for some $\bq\in Gen_3(\bp)$ the 0-1 law for $L$  strongly fails in $M^n_{\bq}$. Motivated by the example mentioned above appearing in the end of section 4 of \cite{LuSh}, we let $\psi$ be the sentence in $L$ implying that each edge of the graph is contained in a cycle of length 4. Once again we use an inductive construction of $(\bq_1,\bq_2,\bq_3,...)$ in $\gP^{fin}$ such that $\bq=\bigcup_{i>0}\bq_i\in Gen_3(\bp)$ and both $\psi$ and $\lnot\psi$ hold infinitely often in $M^n_{\bq}$.
For $i=1$ let $n_{\bq_1}=n_1:=\min\{l:p_l=1\}+1$ and define $(q_1)_l=0$ if $0<l<n_1-1$ and $(q_1)_{n_1-1}=1$.
For even $i>1$ let $n_{\bq_i}=n_i:=\min\{l>4n_{i-1}:p_l=1\}+1$ and define $(q_i)_l=(q_{i-1})_l$ if $0<l< n_{i-1}$, $(q_i)_l=0$ if $n_{i-1}\leq l<n_i-1$ and $(q_1)_{n_1-1}=1$.
For odd $i>i$ recall $n_1=\min\{l:p_l=1\}+1$ and let $n_{\bq_i}=n_i:=n_{i-1}+n_1$. Now define $(q_i)_l=(q_{i-1})_l$ if $0<l< n_{i-1}$ and $(q_i)_l=0$ if $n_{i-1}\leq l< n_i$.
Clearly we have for even $i>1$, $Pr[M^{n_i+1}_{\bq_{n_i+1}}\models\psi]=0$ and for odd $i>1$ $Pr[M^{n_i}_{\bq_{n_i}}\models\psi]=1$. Note that indeed $\bigcup_{i>0}\bq_i\in Gen_3(\bp)$, hence we are done.
\end{proof}

\begin{rem}\label{RemNew}
In the proof of the failure of the convergence law in the case $j=1$ the assumption $|U^*(\bp)|=\infty$ is not needed, our proof works under the weaker assumption $|U^*(\bp)|\geq 2$ and for some $p>0$, $\{l>0:p_l>p\}$ is infinite. See below more on the case $j=1$ and $1<|U^*(\bp)|<\infty$.
\end{rem}

\begin{lem} \label{Lemj=1}
Let $\bq\in\gP^{inf}$ and assume:
\begin{enumerate}
    \item
    Let $l^*<l^{**}$ be the first two members of $U^*(\bq)$ (in particular assume $|U^*(\bq)|\geq 2$) then $l^{**}=2\cdot l^*$.
    \item
    If $l_1,l_2,l_3$ all belong to $\{l>0:q_l>0\}$ and $l_1+l_2=l_3$ then $\{l_1,l_2,l_3\}=\{l,l+l^*,l+l^{**}\}$ for some $l\geq 0$.
    \item
    Let $l^{***}$ be the first member of $\{l>0:0<q_l<1\}$ (in particular assume $|\{l>0:0<q_l<1\}|\geq 1$) then the set $\{n\in\bN:n\leq l \leq n+l^{**}+l^{***}\Rightarrow q_{l}=0\}$ is infinite.
\end{enumerate}
Then the 0-1 law for $L$ fails for $M^n_{\bq}$.
\end{lem}

\begin{proof}
The proof is similar to the case $j=1$ in the proof of Theorem \ref{ThmInf}, hence we will not go into detail. Below $n$ is some large enough natural number (say larger than $3\cdot l^{**}\cdot l^{***}$) such that (3) above holds, and if we say that some property holds in $M^n_{\bq}$ we mean it holds there with probability $1$. Let $\psi^1_{ext}(x)$ be the formula in $L$ implying that $x$ belongs to at most two distinct triangles. Then for all $m\in[n]$:
$$M^n_{\bq}\models \psi^1_{ext}[m] \text{ iff } m\in[1,l^{**}]\cup(n-l^{**},n].$$
Similarly for any natural $t<n/3l^{**}$ define (using induction on $t$): $$\psi^t_{ext}(x):=(\exists y \exists z) x\thicksim y\wedge x\thicksim z\wedge y\thicksim z\wedge (\psi^{t-1}_{ext}(y) \vee \psi^{t-1}_{ext}(z))$$ we then have for all $m\in[n]$:
$$M^n_{\bq}\models \psi^t_{ext}[m] \text{ iff } m\in[1,t l^{**}]\cup(n-t l^{**},n].$$
Now for $1\leq t <n/3l^{**}$ let $m^*(t)$ be the minimal number of edges in $M^n_{\bq}|_{[1,t\cdot l^{**}]\cup(n-t\cdot l^{**},n]}$ i.e only edges with probability one and within one of the intervals are counted, formally $$m^*(t):=2\cdot|\{(m,m'):m<m'\in[1,t\cdot l^{**}]\text{ and }q_{m'-m}=1\}|.$$ Let $1\leq t^* <n/3l^{**}$ be such that $l^{***}<l^{**}\cdot t^*$ (it exists as $n$ is large enough). Note that $m^*(t^*)$ depends only on $\bq$ and not on $n$ hence we can define $$\psi:=\text{"There exists exactly }m^*(t^*)\text{ couples }\{x,y\}\text{ s.t. }\psi^{t^*}_{ext}(x)\wedge\psi^{t^*}_{ext}(y)\wedge x\thicksim y."$$ We then have $Pr[m^n_{\bq}\models\psi]\leq (1-q_{l^{***}})^2<1$ as we have $m^*(t^*)$ edges on $[1,t^* l^{**}]\cup(n-t^* l^{**},n]$ that exist with probability $1$, and at least two additional edges (namely $\{1,l^{***}+1\}$ and $\{n-l^{***},n\}$) that exist with probability $q_{l^{***}}$ each. On the other hand if we define:
$$p':=\prod\{1-q_{m'-m}:m<m'\in[1,t^*\cdot l^{**}]\text{ and }q_{m'-m}<1\}$$ and note that $p'$ does not depend on $n$, then (recalling assumption (3) above) we have $Pr[m^n_{\bq}\models\psi]\geq (p')^2>0$ thus completing the proof.
\end{proof}

\begin{lem} \label{Leml1l2}
Let $\bq\in\gP^{inf}$ be such that for some $l_1<l_2\in\bN\setminus\{0\}$ we have: $0<p_{l_1}<1$, $p_{l_2}=1$ and $p_l=0$ for all $l\not\in\{l_1,l_2\}$. Then the 0-1 law for $L$ fails for $M^n_{\bq}$.
\end{lem}

\begin{proof}
Let $\psi$ be the sentence in $L$ "saying" that some vertex has exactly one neighbor and this neighbor has at least three neighbors. Formally: $$\psi:=(\exists x) (\exists !y)x\thicksim y\wedge(\forall z) x \thicksim z\to(\exists u_1\exists u_2\exists u_3)\bigwedge_{0<i<j\leq3}u_i\neq u_j\wedge\bigwedge_{0<i\leq3}z \thicksim u_i.$$
We first show that for some $p>0$ and $n_0\in\bN$, for all $n>n_0$ we have $Pr[M^n_{\bq}\models\psi]>p$. To see this simply take $n_0=l_1+l_2+1$ and $p=(1-p_{l_1})(p_{l_1})$. Now for $n>n_0$ in $M^n_{\bq}$, with probability $1-p_{l_1}$ the node $1\in[n]$ has exactly one neighbor (namely $1+l_2\in[n]$) and with probability at least $p_{l_1}$, $1+l_2$ is connected to $1+l_1+l_2$, and hence has three neighbors ($1$, $1+2l_2$ and $1+l_1+l_2$). This yields the desired result. On the other hand for some $p'>0$ we have for all $n\in\bN$, $Pr[M^n_{\bq}\models\lnot\psi]>p'$. To see this note that for all $n$, only members of $[1,l_2]\cup(n-l_2,n]$ can possibly exemplify $\psi$, as all members of $(l_2,n-l_2]$ have at least two neighbors with probability one. For each $x\in[1,l_2]\cup(n-l_2,n]$, with probability at least $(1-p_1)^2$, $x$ dose not exemplify $\psi$ (since the unique neighbor of $x$ has less then three neighbors). As the size of $[1,l_2]\cup(n-l_2,n]$ is $2\cdot l_2$ we get $Pr[M^n_{\bq}\models\lnot\psi]>(1-p_1)^{2l_2}:=p'>0$. Together we are done.
\end{proof}

\begin{lem} \label{Lem01}
Let $\bp\in\gP^{inf}$ be such that $|U^*(\bp)|<\infty$ and $p_i\in\{0,1\}$ for $i>0$. Then $M^n_{\bp}$ satisfy the 0-1 law for $L$.
\end{lem}

\begin{proof}
Let $S^n$ be the (not random) structure in vocabulary $\{Suc\}$, with universe $[n]$ and $Suc$ is the successor relation on $[n]$. It is straightforward to see that any sentence $\psi\in L$ has a sentence $\psi^S\in\{Suc\}$ such that
\begin{displaymath}
Pr[M^n_{\bp}\models\psi]=
\left\{ \begin{array}{ll}
      1&S^n\models\psi^S\\
      0&S^n\not\models\psi^S.
\end{array}\right.
\end{displaymath}
Also by a special case of Gaifman's result from \cite{Ga} we have: for each $k\in\bN$ there exists some $n_k\in\bN$ such that if $n,n'>n_k$ then $S^n$ and $S^{n'}$ have the same first order theory of quantifier depth $k$. Together we are done.
\end{proof}

\begin{conc} \label{ConcFin}
Let $\bp\in\gP^{inf}$ be such that $0<|U^*(\bp)|<\infty$.
\begin{enumerate}
    \item The 2-hereditary 0-1 law holds for $\bp$ iff $|\{l>0:p_l>0\}|>1$.
    \item The 3-hereditary 0-1 law holds for $\bp$ iff $\{l>0:0<p_l<1\}\neq\emptyset$.
    \item If furthermore $1<|U^*(\bp)|$ then the 1-hereditary 0-1 law holds for $\bp$ iff $\{l>0:0<p_l<1\}\neq\emptyset$.
\end{enumerate}
\end{conc}

\begin{proof}
For (1) note that if indeed $|\{i>0:p_l>0\}|>1$ then some $\bq\in Gen_2(\bp)$ is as in the assumption of Lemma \ref{Leml1l2}, otherwise any $\bq\in Gen_2(\bp)$ has at most $1$ nonzero member hence $M^n_{\bq}$ satisfy the 0-1 law by either \ref{Lem01} or \ref{ThmPrevL}.

For (2) note that if $\{i>0:0<p_l<1\}\neq\emptyset$ then some $\bq\in Gen_3(\bp)$ is as in the assumption of Lemma \ref{Leml1l2}, otherwise any $\bq\in Gen_3(\bp)$ is as in the assumption of Lemma \ref{Lem01} and we are done.

Similarly for (3) note that if $1<|U^*(\bp)|$ and $\{l>0:0<p_l<1\}\neq\emptyset$ then some $\bq\in Gen_1(\bp)$ satisfies assumptions (1)-(3) of Lemma \ref{Lemj=1}, otherwise any $\bq\in Gen_1(\bp)$ is as in the assumption of Lemma \ref{Lem01} and we are done.
\end{proof}

\section{When exactly one probability equals 1}

In this section we assume:
\begin{assu}
$\bp$ is a fixed member of $\gP^{inf}$ such that $|U^*(\bp)|=1$ hence denote $U^*(\bp)=\{l^*\}$, and assume
$$(*)'\qquad\qquad \lim_{n\to\infty}\log(\prod_{l\in[n]\setminus\{l^*\}}(1-p_l))/\log(n)=0.$$
\end{assu}
We try to determine when the $1$-hereditary 0-1 law holds.
The assumption of $(*)'$ is justified as the proof in section 2 works also in this case and in fact in any case that $U^*(\bp)$ is finite. To see this replace in section 2 products of the form $\prod_{l<n}(1-p_l)$ by $\prod_{l<n,l\not\in U^*(\bp)}(1-p_l)$, sentences of the form "$x$ has valency $m$" by "$x$ has valency $m+2|U^*(\bp)|$", and similar simple changes. So if $(*)'$ fails then the $1$-hereditary weak convergence law fails, and we are done.
It seems that our ability to "identify" the $l^*$-boundary (i.e. the set $[1,l^*]\cup(n-l^*,n]$) in $M^n_{\bp}$ is closely related to the holding of the 0-1 law.
In Conclusion \ref{LemOne1Neg} we use this idea and give a necessary condition on $\bp$ for the $1$-hereditary weak convergence law. The proof uses methods similar to those of the previous sections. Finding a sufficient condition for the $1$-hereditary 0-1 law seems to be harder. It turns out that the analysis of this case is, in a way, similar to the analysis when we add the successor relation to our vocabulary. This is because the edges of the form $\{l,l+l^*\}$ appear with probability $1$ similarly to the successor relation. There are, however, some obvious differences. Let $L^+$ be the vocabulary $\{\thicksim,S\}$, and let ${(M^+)}^n_{\bp}$ be the random $L^+$ structure with universe $[n]$, $\thicksim$ is the same as in $M^n_{\bp}$, and $S^{{(M^+)}^n_{\bp}}$ is the successor relation on $[n]$. Now if for some $l^{**}>0$, $0<p_{l^{**}}<1$ then $({M^+})^n_{\bp}$ does not satisfy the 0-1 law for $L^+$. This is because the elements $1$ and $l^{**}+1$ are definable in $L^+$ and hence some $L^+$ sentence holds in ${(M^+)}^n_{\bp}$ iff $\{1,l^{**}+1\}$ is an edge of ${(M^+)}^n_{\bp}$ which holds with probability $p_{l^{**}}$. In our case, as in $L$ we can not distinguish edges of the form $\{l,l+l^*\}$ from the rest of the edged, the 0-1 law may hold even if such $l^{*}$ exists. In Lemma \ref{LemOne1} below we show that if, in fact, we can not "identify the edges" in $M^n_{\bp}$ then the 0-1 law, holds in $M^n_{\bp}$. This is translated in Theorem \ref{ThmOne1} to a sufficient condition on $\bp$ for the 0-1 law holding in $M^n_{\bp}$, but not necessarily for the $1$-hereditary 0-1 law. The proof uses "local" properties of graphs. It seems that some form of "$1$-hereditary" version of \ref{ThmOne1} is possible. In any case we could not find a necessary and sufficient condition for the $1$-hereditary 0-1 law, and the analysis of this case is not complete.

We first find a necessary condition on $\bp$ for the $1$-hereditary weak convergence law. Let us start with a definition of a structure on a sequence $\bq\in\gP$ that enables us to "identify" the $l^*$-boundary in $M^n_{\bq}$.

\begin{df}
\begin{enumerate}
\item
A sequence $\bq\in\gP$ is called nice if:
\begin{enumerate}
    \item $U^*(\bq)=\{l^*\}$.
    \item If $l_1,l_2,l_3\in\{l<n_{\bq}:q_l>0\}$ then $l_1+l_2\neq l_3$.
    \item If $l_1,l_2,l_3,l_4\in\{l<n_{\bq}:q_l>0\}$ then $l_1+l_2+l_3\neq l_4$.
    \item If $l_1,l_2,l_3,l_4\in\{l<n_{\bp}:q_l>0\}$, $l_1+l_2=l_3+l_4$ and $l_1+l_2<n_{\bq}$ then $\{l_1,l_2\}=\{l_3,l_4\}$.
\end{enumerate}

\item Let $\phi^1$ be the following $L$-formula:
$$\phi^1(y_1,z_1,y_2,z_2):=y_1\thicksim z_1\wedge z_1\thicksim z_2\wedge z_2\thicksim y_2\wedge y_2\thicksim y_1\wedge y_1\neq z_2\wedge z_1\neq y_2.$$

\item For $k\geq 0$ define by induction on $k$ the $L$-formula $\phi_k^1(y_1,z_1,y_2,z_2)$ by:
    \begin{itemize}
    \item $\phi_0^1(y_1,z_1,y_2,z_2):=y_1=y_2\wedge z_1=z_2\wedge y_1\neq z_1$.
    \item $\phi_1^1(y_1,z_1,y_2,z_2):=\phi^1(y_1,z_1,y_2,z_2)$.
    \item $\phi_{k+1}^1(y_1,z_1,y_2,z_2):=\\(\exists y\exists z)[(\phi_k^1(y_1,z_1,y,z)\wedge \phi^1(y,z,y_2,z_2))\vee(\phi_k^1(y_2,z_2,y,z)\phi^1(y_1,z_1,y,z))]$.
    \end{itemize}
\item For $k_1,k_2,\in\bN$ let $\phi_{k_1,k_2}^2$ be the following $L$-formula:
$$\phi_{k_1,k_2}^2(y,z):=(\exists x_1\exists x_2\exists x_3\exists x_4)[\phi_{k_1}^1(y,z,x_1,x_2)\wedge\phi_{k_2}^1(x_2,x_1,x_3,x_4)\wedge \lnot x_3\thicksim x_4].$$
\item For $k_1,k_2,\in\bN$ let $\phi_{k_1,k_2}^3$ be the following $L$ formula:
$$\phi_{k_1,k_2}^3(x):=(\exists !y)[x\thicksim y\wedge\lnot\phi_{k_1,k_2}^2(x,y)].$$
\end{enumerate}
\end{df}

\begin{obs} \label{ObsEx}
Let $\bq\in\gP$ be nice and $n\in\bN$ be such that $n<n_{\bq}$. Then the following holds in $M^n_{\bq}$ with probability $1$:
\begin{enumerate}
    \item For $y_1,z_1,y_2,z_2\in[n]$, if $M^n_{\bq}\models\phi^1[y_1,z_1,y_2,z_2]$ then $y_1-z_1=y_2-z_2$. (Use (d) in the definition of nice).
    \item For $k\in\bN$ and $y_1,z_1,y_2,z_2\in[n]$, if $M^n_{\bq}\models\phi_{k}^1[y_1,z_1,y_2,z_2]$ then $y_1-z_1=y_2-z_2$. (Use (1) above and induction on $k$).
    \item For $k_1,k_2\in\bN$ and $y,z\in[n]$, if $M^n_{\bq}\models\phi_{k_1,k_2}^2[y,z]$ then $|y-z|\neq l^*$. (Use (2) above and the definition of $\phi^2_{k_1,k_2}(y,z)$).
    \item For $k_1,k_2\in\bN$ and $x\in[n]$, if $M^n_{\bq}\models\phi_{k_1,k_2}^3[x]$ then $x\in[1,l^*]\cup(n-l^*,n]$. (Use (3) above).
\end{enumerate}
\end{obs}

The following claim shows that if $\bq$ is nice (and have a certain structure) then, with probability close to $1$, $\phi^3_{3,0}[y]$ holds in $M^n_{\bq}$ for all $y\in[1,l^*]\cup(n-l^*,n]$. This, together with (4) in the observation above gives us a "definition" of the $l^*$-boundary in $M^n_{\bq}$.

\begin{claim} \label{ClaimNiceQ}
Let $\bq\in\gP^{fin}$ be nice and denote $n=n_{\bq}$. Assume that for all $l>0$, $q_l>0$ implies $l<\lfloor n/3 \rfloor$. Assume further that for some $\epsilon>0$, $0<q_l<1\Rightarrow\epsilon<q_l<1-\epsilon$. Let $y_0\in[1,l^*]\cup(n-l^*,n]$. Denote $m:=|\{0<l<n_{\bp}:0<q_l<1\}|$. Then:
$$Pr[M^n_{\bq}\models\lnot\phi^3_{3,0}[y_0]]\leq (\sum_{\{y\in[n]:|y_0-y|\neq l^*\}}q_{|y_0-y|})(1-\epsilon^{11})^{m/2-1}.$$
\end{claim}

\begin{proof}
We deal with the case $y_0\in[1,l^*]$, the case $y_0\in(n-l^*,n]$ is symmetric. Let $z_0\in[n]$ be such that $l_0:=z_0-y_0\in\{0<l<n:0<q_l<1\}$ (so $l_0\neq l^*$ and $l_0<\lfloor n/3\rfloor$), and assume that $M^n_{\bq}\models y_0\thicksim z_0$. For any $l_1,l_2<\lfloor n/3\rfloor$ denote (see diagram below): $y_1:=y_0+l_1$, $y_2:=y_0+l_2$, $y_3:=y_2+l_1=y_1+l_2=y_0+l_1+l_2$ and symmetrically for $z_1, z_2, z_3$ (so $y_i$ and $z_i$ for $i\in\{0,1,2,3\}$ all belong to $[n]$).
\xymatrix{
y_0 \ar@{-}[rrr]^{l_0} \ar@{-}[d]^{l_1} \ar@{-}[rrddd]^{l_2}  & & & z_0 \ar@{-}[d]^{l_1} \ar@{-}[rrddd]^{l_2} & &\\
y_1 \ar@{-}[rrddd]^{l_2} &  & & z_1 \ar@{-}[rrddd]^{l_2} & &\\
&&&&&\\
& & y_2 \ar@{-}[rrr]^{l_0} \ar@{-}[d]^{l_1} & & & z_2 \ar@{-}[d]^{l_1} \\
& & y_3 \ar@{-}[rrr]^{l_0} & & & z_3 }
The following holds in $M^n_{\bq}$ with probability $1$: \underline{If} for some $l_1,l_2<\lfloor n/3\rfloor$ such that $(l_0,l_1,l_2)$ is without repetitions, we have:
\begin{itemize}
    \item[$(*)_1$] $(y_0,y_1,y_3,y_2)$, $(z_0,z_1,z_3,z_2)$ and $(y_2,y_3,z_3,z_2)$ are all circles in $M^n_{\bq}$.
    \item[$(*)_2$] $\{y_1,z_1\}$ is \emph{not} an edge of $M^n_{\bq}$.
\end{itemize}
\underline{Then} $M^n_{\bq}\models\phi^2_{0,3}[y_0,z_0]$. Why? As $(y_1,y_0,z_0,z_1)$, in the place of $(x_1,x_2,x_3,x_4)$, exemplifies  $M^n_{\bp}\models\phi^2_{0,3}[y_0,z_0]$.
Let us fix $z_0=y_0+l_0$ and assume that $M^n_{\bq}\models y_0\thicksim z_0$. (Formally we condition the probability space $M^n_{\bq}$ to the event $y_0\thicksim z_0$.) Denote $$L^{y_0,z_0}:=\{(l_1,l_2):q_{l_1},q_{l_2}>0,l_0\neq l_1,l_0\neq l_2, l_1\neq l_2\}.$$ For $(l_1,l_2)\in L^{y_0,z_0}$, the probability that $(*)_1$ and $(*)_2$ holds, is $(1-q_{l_0})(q_{l_0})^2(q_{l_1})^4(q_{l_2})^4$.
Denote the event that $(*)_1$ and $(*)_2$ holds by $E^{y_0,z_0}(l_1,l_2)$. Note that if $(l_1,l_2),(l'_1,l'_2)\in L^{y_0,z_0}$ are such that $(l_1,l_2,l'_1,l'_2)$ is without repetitions and $l_1+l_2\neq l'_1+l'_2$ then the events $E^{y_0,z_0}(l_1,l_2)$ and $E^{y_0,z_0}(l'_1,l'_2)$ are independent. Now recall that $m:=|\{l>0:\epsilon<q_l<1-\epsilon\}|$. Hence we have some $L'\subseteq L^{y_0,z_0}$ such that: $|L'|=\lfloor m/2-1\rfloor$, and if $(l_1,l_2),(l'_1,l'_2)\in L'$ then the events $E^{y_0,z_0}(l_1,l_2)$ and $E^{y_0,z_0}(l'_1,l'_2)$ are independent. We conclude that $$Pr[M^n_{\bq}\models\lnot\phi^2_{0,3}[y_0,z_0]|M^n_{\bq}\models y_0\thicksim z_0]\leq$$$$(1-(1-q_{l_0})(q_{l_0})^2 (q_{l_1})^4(q_{l_2})^4)^{m/2-1}\leq(1-\epsilon^{11})^{m/2-1}.$$
This is a common bound for all $z_0=y_0+l_0$, and the same bound holds for all $z_0=y_0-l_0$ (whenever it belongs to $[n]$). We conclude that the expected number of $z_0\in[n]$ such that: $|z_0-y_0|\neq l^*$, $M^n_{\bq}\models y_0\thicksim z_0$ and $M^n_{\bq}\models\lnot\phi^2_{0,3}[y_0,z_0]$ is at most  $(\sum_{\{y\in[n]:|y_0-y|\neq l^*\}}q_{|y_0-y|})(1-\epsilon^{11})^{m/2-1}.$
Now by (3) in Observation \ref{ObsEx}, $M^n_{\bq}\models\phi^2_{0,3}[y_0,y_0+l^*]$. By Markov's inequality and the definition of $\phi^3_{0,3}(x)$ we are done.
\end{proof}

We now prove two lemmas which allow us to construct a sequence $\bq$ such that for $\varphi:=\exists x\phi^3_{0,3}(x)$ both $\varphi$ and $\lnot\varphi$ will hold infinitely often in $M^n_{\bq}$.

\begin{lem} \label{LemOne1Pos}
 Assume$\bp$ satisfy $\sum_{l>0}p_l=\infty$, and let $\bq\in Gen_1^r(\bp)$ be nice. Let $\zeta>0$ be some rational number. Then there exists some $r'>r$ and $\bq'\in Gen_1^{r'}(\bp)$ such that: $\bq'$ is nice, $\bq\vartriangleleft\bq'$ and $Pr[M^{n_{\bq'}}_{\bq'}\models\varphi]\leq\zeta$.
\end{lem}

\begin{proof}
Define $p^1:=(\prod_{l\in[n_{\bq}]\setminus\{l^*\}}(1-p_l))^2$, and choose $r'>r$ large enough such that $\sum_{r<l\leq r'}p_l\geq 2l^*\cdot p^1/\zeta$. Now define $\bq'\in Gen_1^{r'}(\bp)$ in the following way:
$$q'_l=
\left\{ \begin{array}{ll}
      q_l&0<l< n_{\bq}\\
      0&n_{\bq}\leq l<(r'-r)\cdot n_{\bq}\\
      p_{r+i}&l=(r'-r+i)\cdot n_{\bq}\textrm{ for some } 0<i\leq(r'-r)\\
      0 & (r'-r)\cdot n_{\bq}\leq l<2(r'-r)\cdot n_{\bq} \textrm{ and } l\not\equiv 0 \pmod {n_{\bq}}.
\end{array}\right.$$
Note that indeed $\bq'$ is nice and $\bq\vartriangleleft\bq'$. Denote $n:=n_{\bq'}=2(r'-r)\cdot n_{\bq}$. Note further that every member of $M^n_{\bq'}$ have at most one neighbor of distance more more than $n/2$, and all the rest of its neighbors are of distance at most $n_{\bq}$. We now bound from above the probability of $M^n_{\bq'}\models\exists x\phi^3_{0,3}(x)$. Let $x$ be in $[1,l^*]$. For each $0<i\leq(r'-r)$ denote $y_i:=x+(r'-r+i)\cdot n_{\bq}$ (hence $y_i\in[n/2,n]$) and let $E_i$ be the following event: "$M^n_{\bq'}\models y_i\thicksim z$ iff $z\in\{x,y_i+l^*,y_i-l^*\}$". By the definition of $\bq'$, each $y_i$ can only be connected to either $x$ of to members of $[y-n_{\bq},y+n_{\bq}]$, hence we have $$Pr[E_i]=q'_{(r'-r+i)\cdot n_{\bq}}\cdot p^1=p_{r+i}\cdot p^1.$$
As $i\neq j\Rightarrow n/2>|y_i-y_j|>n_{\bq}$ we have that the $E_i$-is are independent events. Now if $E_i$ holds then by the definition of $\phi^2_{0,3}$ we have $M^n_{\bq'}\models\lnot\phi^2_{0,3}[x,y_i]$, and as $M^n_{\bq'}\models\lnot\phi^2_{0,3}[x,x+l^*]$ this implies $M^n_{\bq'}\models\lnot\phi^3_{0,3}[x]$. Let the random variable $X$ denote the number of $0<i\leq(r'-r)$ such that $E_i$ holds in $M^n_{\bq'}$. Then by Chebyshev's inequality we have:
$$Pr[M^n_{\bq'}\models\phi^3_{0,3}[x]]\leq Pr[X=0]\leq \frac{Var(X)}{Exp(X)^2}\leq \frac{1}{Exp(X)}\leq \frac{p^1}{\displaystyle{\sum_{0<i\leq(r'-r)}}p_{r+i}}\leq \frac{\zeta}{2l^*}.$$
This is true for each $x\in[1,l^*]$ and the symmetric argument gives the same bound for each $x\in(n-l^*,n]$. Finally note that if $x,x+l^*$ both belong to $[n]$ then $M^n_{\bq'}\models\lnot\phi^2_{0,3}[x,x+l^*]$ (see \ref{ObsEx}(4)). Hence if $x\in(l^*,n-l^*]$ then $M^n_{\bq'}\models\lnot\phi^3_{0,3}[x]$. We conclude that:
$$Pr[M^n_{\bq'}\models\exists x\phi^3_{0,3}(x)]=Pr[M^n_{\bq'}\models\phi]\leq\zeta$$
as desired.
\end{proof}

\begin{lem} \label{LemOne1Neg}
Assume $\bp$ satisfy $0<p_l<1\Rightarrow\epsilon<p_l<1-\epsilon$ for some $\epsilon>0$, and $\sum_{n=1}^{\infty}p_n=\infty$. Let $\bq\in Gen_1^r(\bp)$ be nice, and $\zeta>0$ be some rational number. Then there exists some $r'>r$ and $\bq'\in Gen_1^{r'}(\bp)$ such that: $\bq'$ is nice, $\bq\vartriangleleft\bq'$ and $Pr[M^{n_{\bq'}}_{\bq'}\models\varphi]\geq1-\zeta$.
\end{lem}

\begin{proof}
This is a direct consequence of Claim \ref{ClaimNiceQ}. For each $r'>r$ denote $m(r'):=|\{0<l\leq r': 0<p_l<1\}|$. Trivially we can choose $r'>r$ such that $m(r')(1-\epsilon^{11})^{m(r')/2-1}\leq\zeta$. As $\bq$ is nice there exists some nice $\bq'\in Gen_1^{r'}(\bp)$ such that $\bq\vartriangleleft\bq'$. Note that $$\sum_{\{y\in[n]:|1-y|\neq l^*\}}q'_{|1-y|}\leq\sum_{\{0<l<n_{\bq'}:l\neq l^*\}}q'_l\leq m(r')$$
and hence by \ref{ClaimNiceQ} we have: $$Pr[M^n_{\bq'}\models\lnot\phi]\leq Pr[M^n_{\bq'}\models\lnot\phi^3_{2,0}[1]]\leq m(r')(1-\epsilon^{11})^{m(r')/2-1}\leq \zeta$$
as desired.
\end{proof}

\relax From the last two lemmas we conclude:
\begin{conc} \label{ConcOne1}
Assume that $\bp$ satisfy $0<p_l<1\Rightarrow\epsilon<p_l<1-\epsilon$ for some $\epsilon>0$, and $\sum_{n=1}^{\infty}p_n=\infty$. Then $\bp$ does not satisfy the $1$-hereditary weak convergence law for $L$.
\end{conc}

The proof is by inductive construction of $\bq\in Gen_1(\bp)$ such that for $\varphi:=\exists x\phi^3_{0,3}(x)$ both $\varphi$ and $\lnot\varphi$ hold infinitely often in $M^n_{\bq}$, using Lemmas \ref{LemOne1Pos}, \ref{LemOne1Neg} as done on previous proofs.

\relax From Conclusion \ref{ConcOne1} we have a necessary condition on $\bp$ for the $1$-hereditary weak convergence law. We now find a sufficient condition on $\bp$ for the (not necessarily $1$-hereditary) 0-1 law.
Let us start with definitions of distance in graphs and of local properties in graphs.
\begin{df} Let $G$ be a graph on vertex set $[n]$.
\begin{enumerate}
    \item
    For $x,y\in[n]$ let $dist^G(x,y):=\min\{k\in\bN:G\text{ has a path of length }k\text{ from }x\text{ to }y\}$.
    Note that for each $k\in\bN$ there exists some $L$-formula $\theta_k(x,y)$ such that for all $G$ and $x,y\in[n]$:
    $$G\models\theta_k[x,y]\quad\text{iff}\quad dist^G(x,y)\leq k.$$
    \item
    For $x\in[n]$ and $r\in\bN$ let $B^G(r,x):=\{y\in[n]:dist^G(x,y)\leq r\}$ be the ball with radius $r$ and center $x$ in $G$.
    \item
    An $L$-formula $\phi(x)$ is called $r$-local if every quantifier in $\phi$ is restricted to the set $B^G(r,x)$. Formally each appearance of the form $\forall y...$ in $\phi$ is of the form $(\forall y)\theta_r(x,y)\to...$, and similarly for $\exists y$ and other variables. Note that for any $G$, $x\in[n]$, $r\in\bN$ and an $r$-local formula $\phi(x)$ we have:
      $$ G\models\phi[x]\quad\text{iff}\quad G|_{B(r,x)}\models\phi[x].$$
    \item
    An $L$-sentence is called \emph{local} if it has the form
    $$\exists x_1...\exists x_m\bigwedge_{1\leq i\leq m}\phi(x_i)\bigwedge_{1\leq i<j\leq m}\lnot\theta_{2r}(x_i,x_j)$$
    where $\phi=\phi(x)$ is an $r$-local formula for some $r\in\bN$.
    \item
    For $l,r\in\bN$ and an $L$-formula $\phi(x)$ we say that the $l$-boundary of $G$ is $r$-indistinguishable by $\phi(x)$ if for all $z\in[1,l]\cup(n-l,n]$ there exists some $y\in[n]$ such that $B^G(r,y)\cap([1,l]\cup(n-l,n])=\emptyset$ and $G\models\phi[z]\leftrightarrow\phi[y]$
\end{enumerate}
\end{df}

We can now use the following famous result from \cite{Ga}:
\begin{thm}[\textbf{Gaifman's Theorem}]
Every $L$-sentence is logically equivalent to a boolean combination of local $L$-sentences.
\end{thm}

We will use Gaifman's theorem to prove:
\begin{lem} \label{LemOne1}
Assume that for all $k\in\bN$ and $k$-local $L$-formula $\varphi(z)$ we have:
$$\lim_{n\to\infty}Pr[\text{The }l^*\text{-boundary of }M^n_{\bp}\text{ is }k\text{-indistinguishable by }\varphi(z)]=1.$$
Then the 0-1 law for $L$ holds in $M^n_{\bp}$.
\end{lem}

\begin{proof}
By Gaifman's theorem it is enough if we prove that the 0-1 law holds in $M^n_{\bp}$ for local $L$-sentences. Let
$$\psi:=\exists x_1...\exists x_m\bigwedge_{1\leq i\leq m}\phi(x_i)\bigwedge_{1\leq i<j\leq m}\lnot\theta_{2r}(x_i,x_j)$$ be some local $L$-sentence, where $\phi(x)$ is an $r$-local formula.

Define $\gH$ to be the set of all $4$-tuples $(l,U,u_0,H)$ such that: $l\in\bN$, $U\subseteq[l]$, $u_0\in U$ and $H$ is a graph with vertex set $U$.
We say that some $(l,U,u_0,H)\in\gH$ is $r$-proper for $\bp$ (but as $\bp$ is fixed we usually omit it) if it satisfies:
\begin{itemize}
    \item[$(*_1)$] For all $u\in U$, $dist^H(u_0,u)\leq r$.
    \item[$(*_2)$] For all $u\in U$, if $dist^H(u_0,u) < r$ then $u+l^*,u-l^* \in U$.
    \item[$(*_3)$] $Pr[M^l_{\bp}|_U=H]>0$.
\end{itemize}
We say that a member of $\frak{H}$ is proper if it is $r$-proper for some $r\in\bN$.

Let $H$ be a graph on vertex set $U\subseteq[l]$ and $G$ be a graph on vertex set $[n]$. We say that $f:U\to[n]$ is a strong embedding of $H$ in $G$ if:
\begin{itemize}
    \item $f$ in one-to one.
    \item For all $u,v\in U$, $H\models u\thicksim v$ iff $G\models f(u)\thicksim f(v)$.
    \item For all $u,v\in U$, $f(u)-f(v)=u-v$.
    \item If $i\in Im(f)$, $j\in[n]\setminus Im(f)$ and $|i-j|\neq l^*$ then $G\models\lnot i\thicksim j$.
\end{itemize}
We make two observations which follow directly from the definitions:
\begin{enumerate}
    \item
    If $(l,U,u_0,H)\in \gH$ is $r$-proper and $f:U\to[n]$ is a strong embedding of $H$ in $G$ then $Im(f)=B^G(r,f(u_0))$. Furthermore for any $r$-local formula $\phi(x)$ and $u\in U$ we have, $G\models\phi[f(u)]$ iff $H\models\phi[u]$.
    \item
    Let $G$ be a graph on vertex set $[n]$ such that $Pr[M^n_{\bp}=G]>0$, and $x\in[n]$ be such that $B^G(r-1,x)$ is disjoint to $[1,l^*]\cup(n-l^*,n]$. Denote by $m$ and $M$ the minimal and maximal elements of $B^G(r,x)$ respectively. Denote by $U$ the set $\{i-m+1:i\in B^G(r,x)\}$ and by $H$ the graph on $U$ defined by $H\models u\thicksim v$ iff $G\models (u+m-1)\thicksim (v+m-1)$. Then the $4$-tuple $(M-m+1,U,x-m+1,H)$ is an $r$-proper member of $\gH$. Furthermore for any $r$-local formula $\phi(x)$ and $u\in U$ we have, $G\models\phi[u-m+1]$ iff $H\models\phi[u]$.
\end{enumerate}
We now show that for any proper member of $\frak{H}$ there are many disjoint strong embeddings into $M^n_{\bp}$. Formally:
\begin{claim} \label{ClaimOne1}
Let $(l,U,u_0,H)\in\gH$ be proper, and $c>1$ be some fixed real. Let $E^n_c$ be the following event on $M^n_{\bp}$: "For any interval $I\subseteq[n]$ of length at least  $n/c$ there exists some $f:U\to I$ a strong embedding of $H$ in $M^n_{\bp}$". Then $$\lim_{n\to\infty}Pr[E^n_c\text{ holds in }M^n_{\bp}]=1.$$
\end{claim}

We skip the proof of this claim an almost identical lemma is proved in \cite{LuSh} (see Lemma at page 8 there).

We can now finish the proof of Lemma \ref{LemOne1}. Recall that $\phi(x)$ is am $r$-local formula. We consider two possibilities. First assume that for some $r$-proper $(l,U,u_0,H)\in\gH$ we have $H\models\phi[u_0]$. Let $\zeta>0$ be some real. Then by the claim above, for $n$ large enough, with probability at least $1-\zeta$ there exists $f_1,...,f_m$ strong embeddings of $H$ into $M^n_{\bp}$ such that $\langle Im(f_i):1\leq i\leq m\rangle$ are pairwise disjoint. By observation (1) above we have:
\begin{itemize}
    \item For $1\leq i<j\leq m$, $B^{M^n_{\bp}}(r,f_i(u_0))\cap B^{M^n_{\bp}}(r,f_j(u_0))=\emptyset$.
    \item For $1\leq i \leq m$, $M^n_{\bp}\models\phi[f_i(u_0)]$.
\end{itemize}
Hence $f_1(u_0),...,f_m(u_0)$ exemplifies $\psi$ in $M^n_{\bp}$, so $Pr[M^n_{\bp}\models\psi]\geq 1-\zeta$ and as $\zeta$ was arbitrary we have $\lim_{n\to\infty}Pr[M^n_{\bp}\models\psi]=1$ and we are done.

Otherwise assume that for all $r$-proper $(l,U,u_0,H)\in\gH$ we have $H\models\lnot\phi[u_0]$. We will show that $\lim_{n\to\infty}Pr[M^n_{\bp}\models\psi]=0$ which will finish the proof. Towards contradiction assume that for some $\epsilon>0$ for unboundedly many $n\in\bN$ we have $Pr[M^n_{\bp}\models\psi]\geq\epsilon$. Define the $L$-formula:
$$\varphi(z):=(\exists x) (\theta_{r-1}(x,z)\wedge \phi(x)).$$
Note that $\varphi(z)$ is equivalent to a $k$-local formula for $k=2r-1$. Hence by the assumption of our lemma for some (large enough $n\in\bN$) we have with probability at least $\epsilon/2$: $M^n_{\bp}\models\psi$ and the $l^*$-boundary of $M^n_{\bp}$ is $k$-indistinguishable by $\varphi(z)$. In particular for some $n\in\bN$ and $G$ a graph on vertex set $[n]$ we have:
\begin{itemize}
    \item[$(\alpha)$] $Pr[M^n_{\bp}=G]>0$.
    \item[$(\beta)$] $G\models\psi$.
     \item[$(\gamma)$]The $l^*$-boundary of $G$ is $k$-indistinguishable by $\varphi(z)$.
\end{itemize}
By $(\beta)$ for some $x_0\in[n]$ we have $G\models\phi[x_0]$. If $x_0$ is such that  $B^G(r-1,x_0)$ is disjoint to $[1,l^*]\cup(n-l^*,n]$ then by $(\alpha)$ and observation (2) above we have some $r$-proper $(l,U,u_0,H)\in\gH$ such that $H\models\phi[u_0]$ in contradiction to our assumption. Hence assume that $B^G(r-1,x_0)$ is not disjoint to $[1,l^*]\cup(n-l^*,n]$ and let $z_0\in[n]$ belong to their intersection. So by the definition of $\varphi(z)$ we have $G\models\varphi[z_0]$ and by $(\gamma)$ we have some $y_0\in[n]$ such that $B^G(k,y_0)\cap([1,l^*]\cup(n-l^*,n])=\emptyset$ and $G\models\varphi[y_0]$. Again by the definition of $\varphi(z)$, and recalling that $k=2r-1$ we have some $x_1\in[n]$ such that $B^G(r-1,x_1)\cap([1,l^*]\cup(n-l^*,n])=\emptyset$ and $G\models\phi[x_1]$. So again by $(\alpha)$ and observation (2) we get a contradiction.
\end{proof}

\begin{rem}
Lemma \ref{LemOne1} above gives a sufficient condition for the 0-1 law. If we are only interested in the convergence law, then a weaker condition is sufficient, all we need is that the probability of any local property holding in the $l^*$-boundary converges. Formally:

Assume that for all $r\in\bN$ and $r$-local $L$-formula, $\phi(x)$, and for all $1\leq l\leq l^*$ we have: Both
$\langle Pr[M^n_{\bp}\models\phi[l]:n\in\bN\rangle$ and $\langle Pr[M^n_{\bp}\models\phi[n-l+1]:n\in\bN\rangle$ converge to a limit.
Then $M^n_{\bp}$ satisfies the convergence law.

The proof is similar to the proof of Lemma \ref{LemOne1}. A similar proof on the convergence law in graphs with the successor relation is Theorem 2(i) in \cite{LuSh}.
\end{rem}

We now use \ref{LemOne1} to get a sufficient condition on $\bp$ for the 0-1 law holding in $M^n_{\bp}$. Our proof relays on the assumption that $M^n_{\bp}$ contains few circles, and only those that are "unavoidable". We start with a definition of such circles:

\begin{df} \label{Def:Cycle}Let $n\in\bN$.
\begin{enumerate}
    \item
    For a sequence $\bar{x}=(x_0,x_1,...,x_k)\subseteq [n]$ and $0\leq i<k$ denote $l^{\bar{x}}_i:={x_{i+1}-x_i}$.
    \item
    A sequence $(x_0,x_1,...,x_k)\subseteq [n]$ is called possible for $\bp$ (but as $\bp$ is fixed we omit it and similarly below) if for each $0\leq i<k$, $p_{|l^{\bar{x}}_i|}>0$.
    \item
    A sequence $(x_0,x_1,...,x_k)$ is called a circle of length $k$ if $x_0=x_k$ and $\langle\{x_i,x_{i+1}\}:0\leq i<k\rangle$ is without repetitions.
    \item
    A circle of length $k$, is called simple if $(x_0,x_1,...,x_{k-1})$ is without repetitions.
    \item
    For $\bar{x}=(x_0,x_1,...,x_k)\subseteq [n]$, a pair $(S\dcup A)$ is called a symmetric partition of $\bar{x}$ if:
    \begin{itemize}
        \item $S\dcup A=\{0,...,k-1\}$.
        \item If $i\neq j$ belong to $A$ then $l^{\bar{x}}_i+l^{\bar{x}}_j\neq 0$.
        \item The sequence $\langle l^{\bar{x}}_i:i\in S\rangle$ can be partitioned into two sequences of length $r=|S|/2$: $\langle l_i:0\leq i<r\rangle$ and $\langle l'_i:0\leq i<r\rangle$ such that $l_i+l'_i=0$ for each $0\leq i<r$.
    \end{itemize}
    \item
    For $\bar{x}=(x_0,x_1,...,x_k)\subseteq [n]$ let $(Sym(\bar{x}),Asym(\bar{x}))$ be some symmetric partition of $\bar{x}$ (say the first in some prefixed order).
    Denote $Sym^+(\bar{x}):=\{i\in Sym(\bar{x}):l^{\bar{x}}_i>0\}$.
    \item
    We say that $\bp$ has no unavoidable circles if for all $k\in\bN$ there exists some $m_k\in\bN$ such that if $\bar{x}$ is a \underline{possible} circle of length $k$ then for each $i\in Asym(\bar{x})$, $|l^{\bar{x}}_i|\leq m_k$.
\end{enumerate}
\end{df}

\begin{thm} \label{ThmOne1}
Assume that $\bp$ has no unavoidable circles, $\sum_{l=1}^{\infty}p_l=\infty$ and $\sum_{l=1}^{\infty}(p_l)^2<\infty$.
Then $M^n_{\bp}$ satisfies the 0-1 law for $L$.
\end{thm}

\begin{proof}
Let $\phi(x)$ be some $r$-local formula, and $j^*$ be in $\{1,2,...,l^*\}\cup\{-1,-2,...,-l^*\}$. For $n\in\bN$ let $z_n^*=z^*(n,j^*)$ equal $j^*$ if $j^*>0$ and $n-j^*+1$ if $j^*<0$ (so $z_n^*$ belongs to $[1,l^*]\cup(n-l^*,n]$). We will show that with probability approaching $1$ as $n\to\infty$ there exists some $y^*\in[n]$ such that $B^{M^n_{\bp}}(r,y^*)\cap([1,l^*]\cup(n-l^*,n])=\emptyset$ and $M^n_{\bp}\models\phi[z_n^*]\leftrightarrow\phi[y^*]$. This will complete the proof by Lemma \ref{LemOne1}. For simplicity of notation assume $j^*=1$ hence $z_n^*=1$ (the proof of the other cases is similar). We use the notations of the proof of \ref{LemOne1}. In particular recall the definition of the set $\frak{H}$ and of an $r$-proper member of $\frak{H}$. Now if for two $r$-proper members of $\frak{H}$, $(l^1,x^1,U^1,H^1)$ and $(l^2,x^2,U^2,H^2)$ we have $H^1\models\phi[x^1]$ and $H^2\models\lnot\phi[x^2]$ then by Claim \ref{ClaimOne1} we are done. Otherwise all $r$-proper members of $\frak{H}$ give the same value to $\phi[x]$ and without loss of generality assume that if $(l,x,U,H)\in\frak{H}$ is a $r$-proper then $H\models\phi[x]$ (the dual case is identical). If $\lim_{n\to\infty}Pr[M^n_{\bp}\models\phi[1]]=1$ then again we are done by \ref{ClaimOne1}. Hence we may assume that:

\begin{itemize}
\item[$\odot$]  For some $\epsilon>0$, for an unbounded set of $n\in\bN$, $Pr[M^n_{\bp}\models\lnot\phi[1]]\geq\epsilon$.
\end{itemize}

In the construction below we use the following notations: $2$ denotes the set $\{0,1\}$. ${}^k2$ denotes the set of sequences of length $k$ of members of $2$, and if $\eta$ belongs to ${}^k2$ we write $|\eta|=k$. ${}^{\leq k}2$ denotes $\bigcup_{0\leq i\leq k}{}^k2$ and similarly ${}^{<k}2$. $\langle\rangle$ denotes the empty sequence, and for $\eta,\eta'\in{}^{\leq k}2$, $\eta\hat{}\eta'$ denotes the concatenation of $\eta$ and $\eta'$. Finally for $\eta\in{}^k2$ and $k'<k$, $\eta|_{k'}$ is the initial segment of length $k'$ of $\eta$.

Call $\bar{y}$ a saturated tree of depth $k$ in $[n]$ if:
\begin{itemize}
    \item $\bar{y}=\langle y_{\eta}\in[n]:\eta\in{}^{\leq k}2\rangle$.
    \item $\bar{y}$ is without repetitions.
    \item $\{y_{\langle 0\rangle},y_{\langle 1\rangle}\}=\{y_{\langle\rangle}+l^*,y_{\langle\rangle}-l^*\}$.
    \item If $0<l<k$ and $\eta\in{}^{l}2$ then $\{y_{\eta}+l^*,y_{\eta}-l^*\}\subseteq\{y_{\eta\hat{}\langle 0\rangle},y_{\eta\hat{}\langle 1\rangle},y_{\eta|_{l-1}}\}$.
\end{itemize}

Let $G$ be a graph with set of vertexes $[n]$, and $i\in[n]$. We say that $\bar{y}$ is a circle free saturated tree of depth $k$ for $i$ in $G$ if:
\begin{itemize}
    \item[(i)] $\bar{y}$ is a saturated tree of depth $k$ in $[n]$.
    \item[(ii)] $G\models i\thicksim y_{\langle\rangle}$ but $|i-y_{\langle\rangle}|\neq l^*$.
    \item[(iii)] For each $\eta\in{}^{<k}2$, $G\models y_{\eta}\thicksim y_{\eta\hat{}\langle 0\rangle}$ and $G\models y_{\eta}\thicksim y_{\eta\hat{}\langle 1\rangle}$.
    \item[(iv)] None of the edges described in (ii),(iii) belongs to a circle of length $\leq 6k$ in $G$.
    \item[(v)] Recalling that $\bp$ have no unavoidable circles let $m_{2k}$ be the one from definition \ref{Def:Cycle}(7). For all $\eta\in{}^{\leq k}2$ and $y\in[n]$ if $G\models y_{\eta}\thicksim y$ and $y\not\in\{y_{\eta\hat{}\langle 0\rangle},y_{\eta\hat{}\langle 1\rangle},y_{\eta|_{l-1}},i\}$ then $|y-y_{\eta}|>m_{2k}$.
\end{itemize}

For $I\subseteq [n]$ we say that $\langle\bar{y}^i:i\in I\rangle$ is a circle free saturated forest of depth $k$ for $I$ in $G$ if:
\begin{itemize}
    \item[(a)] For each $i\in I$, $\bar{y}^i$ is a circle free saturated tree of depth $k$ for $i$ in $G$.
    \item[(b)] As sets $\langle\bar{y}^i:i\in I\rangle$ are pairwise disjoint.
    \item[(c)] If $i_1,i_2\in I$ and $\bar{x}$ is a path of length $k'\leq k$ in $G$ from $y^{i_1}_{\langle\rangle}$ to $i_2$, then for some $j<k'$, $(x_j,x_{j+1})=(y^{i_1}_{\langle\rangle},i_1)$.
\end{itemize}

\begin{claim}
For $n\in\bN$ and $G$ a graph on $[n]$ denote by $I_k^*(G)$ the set $([1,l^*]\cup(n-l^*,n])\cap B^G(1,k)$. Let $E^{n,k}$ be the event: "There exists a circle free saturated forest of depth $k$ for $I^*_k(G)$". Then for each $k\in\bN$:
$$\lim_{n\to\infty}Pr[E^{n,k}\text{ holds in }M^n_{\bp}]=1.$$
\end{claim}

\begin{proof}
Let $k\in\bN$ be fixed. The proof proceeds in six steps:

\textbf{Step 1.} We  observe that only a bounded number of circles starts in each vertex of $M^n_{\bp}$. Formally
For $n,m\in\bN$ and $i\in[n]$ let $E^1_{n,m,i}$ be the event: "More than $m$ different circles of length at most $12k$ include $i$". Then for all $\zeta>0$ for some $m=m(\zeta)$ ($m$ depends also on $\bp$ and $k$ but as those are fixed we omit them from the notation and similarly below) we have:
\begin{itemize}
    \item[$\circledast_1$]  For all $n\in\bN$ and $i\in[n]$, $Pr_{M^n_{\bp}}[E^1_{n,m,i}]\leq \zeta$.
\end{itemize}
To see this note that if $\bar{x}=(x_0,...,x_{k'})$ is a possible circle in $[n]$, then
$$Pr[\bar{x} \text{ is a weak circle in }M^n_{\bp}]:=p(\bar{x})=\prod_{i\in Asym(\bar{x})}p_{|l^{\bar{x}}_i|}\cdot\prod_{i\in Sym^+(\bar{x})}(p_{l^{\bar{x}}_i})^2.$$
Now as $\bp$ has no unavoidable, circles let $m_{12k}$ be as in \ref{Def:Cycle}(7). Then the expected number of circles of length $\leq 12k$ starting in $i=x_0$ is
$$\sum_{\substack{k'\leq 12k,\bar{x}=(x_0,...,x_{k'})\\ \text{is a possible circle}}} p(\bar{x})\leq (m_{12k})^{12k}\cdot\sum_{0<l_1,...,l_{6k}<n}\prod_{i=1}^{6k}(p_{l_i})^2\leq(m_12k)^{12k}\cdot(\sum_{0<l<n}(p_l)^2)^{6k}.$$
But as $\sum_{0<l<n}(p_l)^2$ is bounded by $\sum_{l=1}^{\infty}(p_l)^2:=c^*<\infty$, if we take
$m=(m_{12k})^{12k}\cdot(c^*)^{6k}/{\zeta}$ then we have $\circledast_1$ as desired.

\textbf{Step 2.} We show that there exists a positive lower bound on the probability that a circle passes through a given edge of $M^n_{\bp}$. Formally: Let $n\in\bN$ and $i,j\in[n]$ be such that $p_{|i-j|}>0$. Denote By $E^2_{n,i,j}$ the event: "There does not exists a circle of length $\leq6k$ containing the edge $\{i,j\}$". Then there exists some $q_2>0$ such that:
\begin{itemize}
    \item[$\circledast_2$] For any $n\in\bN$ and $i,j\in[n]$ such that $p_{|i-j|}>0$, $Pr_{M^n_{\bp}}[E^2_{n,i,j}|i\thicksim j]\geq q_2$.
\end{itemize}
To see this call a path $\bar{x}=(x_0,...,x_{k'})$ good for $i,j\in[n]$ if $x_0=j$, $x_{k'}=i$, $\bar{x}$ does not contain the edge $\{i,j\}$ and does not contain the same edge more than once. Let $E'^2_{n,i,j}$ be the event: "There does not exists a path good for $i,j$ of length $<6k$". Note that for $i,j\in[n]$ and $G$ a graph on $[n]$ such that $G\models i\thicksim j$ we have: $(i,j,x_2,...,x_{k'})$ is a circle in $G$ iff
$(j,x_2,...,k_{k'})$ is a path in $G$ good for $i,j$. Hence for such $G$ we have: $E^2_{n,i,j}$ holds in $G$ iff $E'^2_{n,i,j}$ holds in $G$. Since the events $i\thicksim j$ and $E'^2_{n,i,j}$ are independent in $M^n_{\bp}$ we conclude:
$$Pr_{M^n_{\bp}}[E^2_{n,i,j}|i\thicksim j]=Pr_{M^n_{\bp}}[E'^2_{n,i,j}|i\thicksim j]=Pr_{M^n_{\bp}}[E'^2_{n,i,j}].$$
Next recalling Definition \ref{Def:Cycle}(7) let $m_k$ be as there. Since $\sum_{l>0}(p_l)^2<\infty$, $(p_l)^2$ converges to $0$ as $l$ approaches infinity, and hence so does $p_l$. Hence for some $m^0\in\bN$ we have $l>m^0$ implies $p_l<1/2$. Let $m^*_k:=\max\{m_{6k},m^0\}$. We now define for a possible  path $\bar{x}=(x_0,...x_{k'})$, $Large(\bar{x})=\{0\leq r <k':|l^{\bar{x}}_r|>m^*_k\}$. Note that as $\bp$ have no unavoidable circles we have for any possible circle $\bar{x}$ of length $\leq 6k$, $Large(\bar{x})\subseteq Sym(\bar{x})$, and $|Large(\bar{x})|$ is even. We now make the following claim: For each $0\leq k^* \leq \lfloor k/2 \rfloor$ let $E'^{2,k^*}_{n,i,j}$ be the event: "There does not exists a path, $\bar{x}$, good for $i,j$ of length $<6k$ with $|Large(\bar{x})|=2k^*$". Then there exists a positive probability $q_{2,k^*}$ such that for any $n\in\bN$ and $i,j\in[n]$ we have:
$$Pr_{M^n_{\bp}}[E'^{2,k^*}_{n,i,j}]\geq q_{2,k^*}.$$
Then by taking $q_2=\prod_{0\leq k^* \leq \lfloor k/2 \rfloor}q_{2,k^*}$ we will have $\circledast_2$. Let us prove the claim. For $k^*=0$ we have (recalling that no circle consists only of edges of length $l^*$):
\begin{eqnarray*}
Pr_{M^n_{\bp}}[E'^{2,0}_{n,i,j}]&=&\prod_{\substack{k'\leq 6k,\text{ }\bar{x}=(i=x_0,j=x_1,...,x_{k'})\\ \text{is a possible circle, }|Large(\bar{x})|=0}}(1-\prod_{r=1}^{k'-1}p_{|l^{\bar{x}}_r}|)\\
&\geq&(1-\max\{p_l:0<l\leq m^*_k,l\neq l^*\})^{6k\cdot(m^*_k)^{6k-1}}.
\end{eqnarray*}
But as the last expression is positive and depends only on $\bp$ and $k$ we are done.
For $k^*>0$ we have:
\begin{eqnarray*}
Pr_{M^n_{\bp}}[E'^{2,k^*}_{n,i,j}]&=&\prod_{\substack{k'\leq 6k,\text{ }\bar{x}=(i=x_0,j=x_1,...,x_{k'})\\ \text{is a possible circle, }|Large(\bar{x})|=k^*}}(1-\prod_{m=1}^{k'-1}p_{|l^{\bar{x}}_m}|)\\
&=&\prod_{\substack{k'\leq 6k,\text{ }\bar{x}=(i=x_0,j=x_1,...,x_{k'})\\ \text{is a possible circle, }\\|Large(\bar{x})|=k^*\text{ ,}0\not\in Large(\bar{x})}}(1-\prod_{m=1}^{k'-1}p_{|l^{\bar{x}}_m}|)\cdot
\prod_{\substack{k'\leq 6k,\text{ }\bar{x}=(i=x_0,j=x_1,...,x_{k'})\\ \text{is a possible circle, }\\|Large(\bar{x})|=k^*\text{ ,}0\in Large(\bar{x})}}(1-\prod_{m=1}^{k'-1}p_{|l^{\bar{x}}_m}|).
\end{eqnarray*}
But the product on the left of the last line is at least $$[\prod_{l_1,...,l_{k^*}>m^*_k}(1-\prod_{m=1}^{k^*}(p_{l_m})^2)]^{(m^*_k)^{(6k-2k^*)}\cdot(6k)^{2k^*}},$$
and as $\sum_{l>m^*_k}(p_l)^2\leq c^*<\infty$ we have $\sum_{l_1,...,l_{k^*}>m^*_k}\prod_{m=1}^{k^*}(p_{l_m})^2\leq(c^*)^{k^*}<\infty$ and hence $\prod_{l_1,...,l_{k^*}>m^*_k}(1-\prod_{m=1}^{k^*}(p_{l_m})^2)>0$ and we have a bound as desired.
Similarly the product on the right is at least
$$[\prod_{l_1,...,l_{k^*-1}>m^*_k}(1-\prod_{m=1}^{k^*-1}(p_{l_m})^2)\cdot 1/2]^{(m^*_k)^{(6k-2k^*-1)}\cdot(6k)^{2k^*}},$$
and again we have a bound as desired.

\textbf{Step 3.} Denote $$E^3_{n,i,j}:=E^2_{n,i,j}\wedge \bigwedge_{r=1,...,k} (E^2_{n,j+(r-1)l^*,j+rl^*}\wedge E^2_{n,j,j-(r-1)l^*,j-rl^*})$$ and let $q_3=q_2^{(2l^*+1)}$. We then have:
\begin{itemize}
    \item[$\circledast_3$] For any $n\in\bN$ and $i,j\in[n]$ such that $p_{|i-j|}>0$ and $j+kl^*,j-kl^*\in[n]$, $Pr_{M^n_{\bp}}[E^3_{n,i,j}|i\thicksim j]\geq q_3$.
\end{itemize}
This follows immediately from $\circledast_2$, and the fact that if $i,i',j,j'$ all belong to $[n]$ then the probability $Pr_{M^n_{\bp}}[E^2_{n,i,j}|E^2_{n,i',j'}]$ is no smaller then the probability $Pr_{M^n_{\bp}}[E^2_{n,i,j}]$.

\textbf{Step 4.} For $i,j\in[n]$ such that $j+kl^*,j-kl^*\in[n]$ denote by $E^4_{n,i,j}$ the event: "$E^3_{n,i,j}$ holds and for $x\in\{j+rl^*:r\in\{-k,-k+1,...,k\}\}$ and $y\in[n]\setminus\{i\}$ we have $x\thicksim y\Rightarrow (|x-y|=l^*\vee |x-y|>m_{2k})$". Then for some $q_4>0$ we have:
\begin{itemize}
    \item[$\circledast_4$] For any $n\in\bN$ and $i,j\in[n]$ such that $p_{|i-j|}>0$ and $j+kl^*,j-kl^*\in[n]$, $Pr_{M^n_{\bp}}[E^4_{n,i,j}|i\thicksim j]\geq q_4$.
\end{itemize}
To see this simply take $q_4=q_3\cdot (\prod_{l\in\{1,...,m_{2k}\}\setminus\{l^*\}}(1-p_l))^{2k+1}$, and use
$\circledast_3$.

 \textbf{Step 5.} For $n\in\bN$, $S\subseteq[n]$, and $i\in[n]$ let $E^5_{n,S,i}$ be the event: "For some $j\in[n]\setminus S$ we have $i\thicksim j$, $|i-j|\neq l^*$ and $E^4_{n,i,j}$". Then for each $\delta>0$ and $s\in\bN$, for $n\in\bN$ large enough (depending on $\delta$ and $s$) we have:
\begin{itemize}
    \item[$\circledast_5$] For all $i\in[n]$ and $S\subseteq[n]$ with $|S|\leq s$, $Pr_{M^n_{\bp}}[E^5_{n,S,i}]\geq 1-\delta$.
\end{itemize}
First let $\delta>0$ and $s\in\bN$ be fixed.
Second for $n\in\bN$, $S\subseteq[n]$ and $i\in[n]$ denote by $J^{n,S}_i$ the set of all possible candidates for $j$, namely $J^{n,S}_i:=\{j\in(kl^*,n-kl^*]\setminus S:|i-j|\neq l^*\}$. For $j\in J^{n,\emptyset}_i$ let $U_j:=\{j+rl^*:r\in\{-k,-k+1,...,k\}\}$. For $m\in\bN$ and $G$ a graph on $[n]$ call $j\in J^{n,S}_i$ a candidate of type $(n,m,S,i)$ in $G$, if each $j'\in U(j)$, belongs to at most $m$ different circles of length at most $6k$ in $G$. Denote the set of all candidates of type $(n,m,S,i)$ in $G$ by $J^{n,S}_i(G)$. Now let $X^{n,m}_i$ be the random variable on $M^n_{\bp}$ defined by:
$$X^{n,m}_i(M^n_{\bp})=\sum\{p_{|i-j|}:j\in J^{n,S}_i(M^n_{\bp})\}.$$
Denote $R^{n,S}_i:=\sum\{p_{|i-j|}:j\in J^{n,S}_i\}$.
Trivially for all $n,m,S,i$ as above, $X^{n,m}_i\leq R^{n,S}_i$. On the other hand, by $\circledast_1$ and the definition of a candidate, for all $\zeta>0$ we can find $m=m(\zeta)\in\bN$ such that for all $n,S,i$ as above and $j\in J^{n,S}_i$, the probability that $j$ is a candidate of type $(n,m,S,i)$ in $M^n_{\bp}$ is at least $1-\zeta$. Then for such $m$ we have: $Exp(X^{n,m}_i)\geq R^{n,S}_i(1-\zeta)$. Hence we have $Pr_{M^n_{\bp}}[X^{n,m}_i\leq R^{n,S}_i/2]\leq 2\zeta$. Recall that $\delta>0$ was fixed, and let $m^*=m(\delta/4)$. Then for all $n,S,i$ as above we have with probability at least $1-\delta/2$, $X^{n,m^*}_i(M^n_{\bp})\geq R^{n,S}_i/2$.
Now denote $m^{**}:=(2l^*+1)(m^*+2m_{2k})6k(m^*+1)$, and fix $n\in\bN$ such that $\sum_{0<l<n}p_l>2\cdot((m^{**}/(q_4\cdot \delta)\cdot2m_{2k}(2l^*+1)+(s+2kl^*+2))$. Let $i\in[n]$ and $S\subseteq[n]$ be such that $|S|\leq s$. We relatives our probability space $M^n_{\bp}$ to the event $X^{n,m^*}_i(M^n_{\bp})\geq R^{n,S}_i/2$, and all probabilities until the end of Step 5 will be conditioned to this event. If we show that under this assumption we have, $Pr_{M^n_{\bp}}[E^5_{n,S,i}]\geq 1-\delta/2$ then we will have $\circledast_5$.

Let $G$ be a graph on $[n]$ such that, $X^{n,m^*}_i(G)\geq R^{n,S}_i/2$. For $j\in J^{n,S}_i$ let $C_j(G)$ denote the set of all the pairs of vertexes which are relevant for the event $E^4_{n,i,j}$. Namely $C_j(G)$ will contain: $\{i,j\}$, all the edges $\{u,v\}$ such that $:u\in U(j)$, $v\neq i$ and $|u-v|<m_{2k}$, and all the edges that belong to a circle of length $\leq 6k$ containing some member of $U(j)$. We make some observations:
\begin{enumerate}
\item
$X^{n,m^*}_i(G)\geq (m^{**}/(q_4\cdot \delta))\cdot2m_{2k}(2l^*+1)$.
\item
There exists $J^1(G)\subseteq J^{n,S}_i$ such that:
    \begin{enumerate}
    \item The sets $U(j)$ for $j\in J^1(G)$ are pairwise disjoint. Moreover if $j_1,j_2\in J^1(G)$, $u_l\in U(j_l)$ for $l\in\{1,2\}$ and $j_1\neq j_2$ then $|u_1-u_2|>m_{2k}$.
    \item Each $j\in J^1(G)$ is a candidate of type $(n,m^*,S,i)$ in $G$.
    \item The sum $\sum\{p_{|i-j|}:j\in J^1(G)\}$ is at least $m^{**}/(q_4\cdot \delta)$.
    \end{enumerate}
    [To see this use (1) and construct $J^1$ by adding the candidate with the largest $p_{|i-j|}$ that satisfies (a). Note that each new candidate excludes at most $m_{2k}(2l^*+1)$ others.]
\item
Let $j$ belong to $J^1(G)$. Then the set $\{j'\in J^1(G):C_j(G)\cap C_{j'}(G)\neq\emptyset\}$ has size at most $m^{**}$. [To see this use (2)(b) above, the fact that two circles of length $\leq 6k$ that intersect in an edge give a circle of length $\leq 12k$ and similar trivial facts.]
\item
\relax From (3) we conclude that there exists $J^2(G)\subseteq j^1(G)$ and $\langle j_1,...j_r\rangle$ an enumeration of $J^2(G)$ such that:
    \begin{enumerate}
    \item For any $1\leq r'\leq r$ the sets $C(j_{r'})$ and $\cup_{1\leq r'' <r'}C(j_{r''})$ are disjoint.
    \item The sum $\sum\{p_{|i-j|}:j\in J^2(G)\}$ is greater or equal $1/(q_4\cdot \delta)$.
    \end{enumerate}
\end{enumerate}
Now for each $j\in J^{n,S}_i$ let $E^*_j$ be the event: "$i \thicksim j$ and $E^4_{n,i,j}$". By $\circledast_4$ we have for each $j\in J^{n,S}_i$, $Pr_{M^n_{\bp}}[E^*_j]\geq q_4\cdot p_{|i-j|}$. Recall that we condition the probability space $M^n_{\bp}$ to the event $X^{n,m^*}_i(M^n_{\bp})\geq R^{n,S}_i/2$, and let $\langle j_1,...j_r\rangle$ be the enumeration of $J^2(M^n_{\bp})$ from (4) above. (Formally speaking $r$ and each $j_{r'}$ is a function of $M^n_{\bp}$). We then have for $1\leq r'<r''\leq r$, $Pr_{M^n_{\bp}}[E^*_{j_{r'}}|E^*_{j_{r''}}]\geq Pr_{M^n_{\bp}}[E^*_{j_{r'}}]$, and $Pr_{M^n_{\bp}}[E^*_{j_{r'}}|\lnot E^*_{j_{r''}}]\geq Pr_{M^n_{\bp}}[E^*_{j_{r'}}]$. To see this use (2)(a) and (4)(a) above and the definition of $C_j(G)$.

Let the random variables $X$ and $X'$ be defined as follows. $X$ is the number of $j\in J^2(M^n_{\bp})$ such that $E^*_j$ holds in $M^n_{\bp}$. In other words $X$ is the sum of $r$ random variables $\langle Y_1,...,Y_r\rangle$, where for each $1\leq r' \leq r$, $Y_{r'}$ equals $1$ if $E^*_{j_{r'}}$ holds, and $0$ otherwise. $X'$ is the sum of $r$ \underline{independent} random variables $\langle Y'_1,...,Y'_r\rangle$, where for each $1\leq r' \leq r$ $Y'_{r'}$ equals $1$ with probability $q_4\cdot p_{|i-j_{r'}|}$ and $0$ with probability $1-q_4\cdot p_{|i-j_{r'}|}$. Then by the last paragraph for any $0 \leq t \leq r$, $$Pr_{M^n_{\bp}}[X\geq t]\geq Pr[X'\geq t].$$
But $Exp(X')=Exp(X)=q_4\cdot\sum_{1\leq r' \leq r}p_{|i-j_{r'}|}$ and by (4)(b) above this is grater or equal $1/\delta$. Hence by Chebyshev's inequality we have:
$$Pr_{M^n_{\bp}}[\lnot E^5_{n,S,i}]\leq Pr_{M^n_{\bp}}[X=0]\leq Pr[X'=0]\leq \frac{Var(X')}{Exp(X')^2}\leq \frac{1}{Exp(X')}\leq
\delta$$
as desired.

\textbf{Step 6.}
We turn to the construction of the circle free saturated forest. Let $\epsilon>0$, and we will prove that for $n\in\bN$ large enough we have $Pr[E^{n,k}\text{ holds in }M^n_{\bp}]\geq 1-\epsilon$. Let $\delta=\epsilon/(l^*2^{k+2})$ and $s=2l^*((k+2^k)(2l^*k+1))$. Let $n\in\bN$ be large enough such that $\circledast_5$ holds for $n$, $k$, $\delta$ and $s$. We now choose (formally we show that with probability at least $1-\epsilon$ such a choice exists) by induction on $(i,\eta)\in I^*_k(M^n_{\bp})\times {}^{\leq k}2$ (ordered by the lexicographic order) $y^i_{\eta}\in[n]$ such that:
\begin{enumerate}
    \item $\langle y^i_{\eta}\in[n]:(i,\eta)\in I^*_k(M^n_{\bp})\times {}^{\leq k}2\rangle$ is without repetitions.
    \item If $\eta=\langle\rangle$ then $M^n_{\bp}\models i\thicksim y^i_{\eta}$, but $|i-y^i_{\eta}|\neq l^*$.
    \item If $\eta\neq\langle\rangle$ then $M^n_{\bp}\models y^i_{\eta}\thicksim y^i_{\eta|_{|\eta|-1}}$.
    \item If $\eta=\langle\rangle$ then $M^n_{\bp}$ satisfies $E^4_{n,i,y^i_{\eta}}$ else, denoting $\rho:=\eta|_{|\eta|-1}$, $M^n_{\bp}$ satisfies $E^4_{n,y^i_{\rho},y^i_{\eta}}$.
\end{enumerate}
Before we describe the choice of $y^i_{\eta}$, we need to define sets $S^i_{\eta}\subseteq[n]$. For a graph $G$ on $[n]$ and $i\in I^*_k(G)$ let $S^*_i(G)$ be the set of vertexes in the first (in some pre fixed order) path of length $\leq k$ from $1$ to $i$ in $G$. Now let $S^*(G)=\bigcup_{i\in I^*_k(G)}S^*_i(G)$. For $(i,\eta)\in I^*_k(M^n_{\bp})\times {}^{\leq k}2$ and $\langle y^{i'}_{\eta'}\in[n]:(i',\eta')<_{lex}(i,\eta)\rangle$ define: $$S^i_{\eta}(G)=S^*(G)\cup\{[y^{i'}_{\eta'}-kl^*,y^{i'}_{\eta'}+kl^*]:(i'\eta')<_{lex}(i,\eta)\}.$$ Note that indeed $|S^*(G)|\leq s$ for all $G$. In the construction below when we write $S^i_{\eta}$ we mean $S^i_{\eta}(M^n_{\bp})$ where $\langle y^{i'}_{\eta'}\in[n]:(i',\eta')<_{lex}(i,\eta)\rangle$ were already chosen.
Now the choice of $y^i_{\eta}$ is as follows:
\begin{itemize}
    \item If $\eta=\langle\rangle$ by $\circledast_5$ with probability at least $1-\delta$, $E^5_{n,S^i_{\eta},i}$ holds in $M^n_{\bp}$ hence we can choose $y^i_{\eta}$ that satisfies (1)-(4).
    \item If $\eta=\langle0\rangle$ (resp. $\eta=\langle1\rangle$) choose $y^i_\eta=y^i_{\langle\rangle}-l^*$ (resp. $y^i_\eta=y^i_{\langle\rangle}+l^*$). By the induction hypothesis and the definition of $E^4_{n,i,j}$ this satisfies (1)-(4) above.
    \item If $|\eta|>1$, $|y^i_{\eta|_{|\eta|-1}}-y^i_{\eta|_{|\eta|-2}}|\neq l^*$ and $\eta(|\eta|)=0$ (resp. $\eta(|\eta|)=1$) then choose $y^i_\eta=y^i_{\eta|_{|\eta|-1}}-l^*$ (resp. $y^i_\eta=y^i_{\eta|_{|\eta|-1}}+l^*$). Again by the induction hypothesis and the definition of $E^4_{n,i,j}$  this satisfies (1)-(4).
    \item If $|\eta|>1$, $y^i_{\eta|_{|\eta|-1}}-y^i_{\eta|_{|\eta|-2}}= l^*$ (resp. $y^i_{\eta|_{|\eta|-1}}-y^i_{\eta|_{|\eta|-2}}=-l^*$) and $\eta(|\eta|)=0$, then choose $y^i_\eta=y^i_{\eta|_{|\eta|-1}}-l^*$ (resp. $y^i_\eta=y^i_{\eta|_{|\eta|-1}}+l^*)$.
    \item If $|\eta|>1$, $|y^i_{\eta|_{|\eta|-1}}-y^i_{\eta|_{|\eta|-2}}| = l^*$ and $\eta(|\eta|)=1$. Then by $\circledast_5$ with probability at least $1-\delta$, $E^5_{n,S^i_{\eta},y^i_{\eta|_{|\eta|-1}}}$ holds in $M^n_{\bp}$, and hence we can choose $y^i_{\eta}$ that satisfies (1)-(4).
\end{itemize}
At each step of the construction above the probability of "failure" is at most $\delta$, hence with probability at least $1-(l^*2^{k+2})\delta=1-\epsilon$ we compleat the construction. It remains to show that indeed $\langle y^i_{\eta}:i\in I^n,\eta\in{}^{\leq k}2\rangle$ is a circle free saturated forest of depth $k$ for $I^*_k$ in $M^n_{\bp}$. This is straight forward from the definitions. First each $\langle y^i_{\eta}:\eta\in{}^{\leq k}2\rangle$ is a saturated tree of depth $k$ in $[n]$ by its construction. Second (ii) and (iii) in the definition of a saturated tree holds by (2) and (3) above (respectively). Third note that by (4) each edge $(y,y')$ of our construction satisfies $E^2_{n,y,y'}$ and $E^4_{n,y,y'}$ hence (iv) and (v) (respectively) in the definition of a saturated tree follows. Lastly we need to show that (c) in the definition of a saturated forest holds. To see this note that if $i_1,i_2\in i^*_k(M^n_{\bp})$ then by the definition of $S^i_{\eta}(M^n_{\bp})$ there exists a path of length $\leq 2k$ from $i_1$ to $i_2$ with all its vertexes in $S^i_{\eta}(M^n_{\bp})$. Now if $\bar{x}$ is a path of length $\leq k$ from $y^{i_1}_{\langle\rangle}$ to $i_2$ and $(y^{i_1}_{\langle\rangle},i_1)$ is not an edge of $\bar{x}$, then necessarily $\{y^{i_1}_{\langle\rangle},i_1\}$ is included in some circle of length $\leq3k+2$. A contradiction to the choice of $y^{i_1}_{\langle\rangle}$. This completes the proof of the claim.
\end{proof}

By $\odot$ and the claim above we conclude that, for some large enough $n\in\bN$, there exists a graph $G=([n],\thicksim)$ such that:
\begin{enumerate}
\item $G\models\lnot\phi[1]$.
\item $Pr[M^n_{\bp}=G]>0$.
\item There exists $\langle\bar{y}^i:i\in I^*_r(G)\rangle$, a circle free saturated forest of depth $r$ for $I^*_r(G)$ in $G$.
\end{enumerate}

Denote $B=B^G(1,r)$, $I=I^*_r(G)$, and we will prove that for some $r$-proper $(l,{u_0},U,H)\in\frak{H}$ we have $(B,1)\cong(H,{u_0})$ (i.e. there exists a graph isomorphism from $G|_B$ to $H$ mapping $1$ to ${u_0}$). As $\phi$ is $r$-local we will then have $H\models\lnot\phi[{u_0}]$ which is a contradiction of our assumption and we will be done. We turn to the construction of $(l,{u_0},U,H)$. For $i\in I$ let $r(i)=r-dist^G(1,i)$. Denote $$Y:=\{y^{i}_{\eta}:i\in I,\eta\in{}^{< r(i)}2\}.$$ Note that by (ii)-(iii) in the definition of a saturated tree we have $Y\subseteq B$. We first define a one-to-one function $f:B\to \mathbb{Z}$ in three steps:

\textbf{Step 1.} For each $i\in I$ define $$B_{i}:=\{x\in B:\text{ there exists a path of length }\leq r(i)\text{ from } x \text{ to } i \text{ disjoint to }Y\}$$ and $B^0:=I\cup\bigcup_{i\in I}B_{i}$. Now define for all $x\in B^0$, $f(x)=x$. Note that:
\begin{itemize}
    \item[$\bullet_1$] $f|_{B^0}$ is one-to-one (trivially).
    \item[$\bullet_2$] If $x\in B^0$ and $dist^G(1,x)<r$ then $x+l^*\in[n]\Rightarrow x+l^*\in B^0$ and $x-l^*\in[n]\Rightarrow x-l^*\in B^0$ (use the definition of a saturated tree).
\end{itemize}

\textbf{Step 2.} We define $f|_Y$. We start by defining $f(y)$ for $y\in\bar{y}^1$, so let $\eta\in{}^{\leq r}2$ and denote $y=y^{1}_{\eta}$. We define $f(y)$ using induction on $\eta$ were ${}^{\leq r}2$ is ordered by the lexicographic order. First if $\eta=\langle\rangle$ then define $f(y)=1-l^*$. If $\eta\neq\langle\rangle$ let $\rho:\eta|_{|\eta|-1}$, and consider $u:=f(y^{1}_{\rho})$. Denote $F=F_{\eta}:=\{f(y^1_{\eta'}):\eta'<_{lex}\eta\}$. Now if $u-l^*\not\in F$ define $f(y)=u-l^*$. If $u-l^*\in F$ but $u+l^*\not\in F$ define $f(y)=u+l^*$. Finally, if $u-l^*,u+l^*\in F$, choose some $l=l_{\eta}$ such that $p_l>0$ and $u-l<\min{F}-rl^*-n$, and define $f(y)=u-l$. Note that by our assumptions $\{l:p_l>0\}$ is infinite so we can always choose $l$ as desired. Note further that we chose $f(y)$ such that $f|_{\bar{y}^1}$ is one-to-one. Now for each $i\in I\cap[1,l^*]$ and $\eta\in{}^{<r(i)}2$, define $f(y^{i}_{\eta})=f(y^1_{\eta})+(f(i)-1)$ (recall that $f(i)=i$ was defined in Step 1, and that $k(i)\leq k(1)$ so $f(y^{i}_{\eta})$ is well defined). For $i\in I\cap(n-l^*,n]$ preform a similar construction in "reversed directions". Formally define $f(y^{i}_{\langle\rangle})=i+l^*$, and the induction step is similar to the case $i=1$ above only now choose $l$ such that $u+l>\max{F}+rl^*+n$, and define $f(y)=u+l$. Note that:
\begin{itemize}
    \item[$\bullet_3$] $f|_Y$ is one-to-one.
    \item[$\bullet_4$] $f(Y)\cap f(B^0)=\emptyset$. In fact:
    \item[$\bullet_4^+$] $f(Y)\cap[n]=\emptyset$.
    \item[$\bullet_5$] If $i\in I\cap[1,l^*]$ then $i-l^*\in f(Y)$ (namely $i-l^*=f(y^i_{\langle\rangle})$).
    \item[$\bullet_5'$] If $i\in I\cap(n-l^*,n]$ then $i+l^*\in f(Y)$ (namely $i+l^*=f(y^i_{\langle\rangle})$).
     \item[$\bullet_6$] If $y\in Y\setminus\{y^{i}_{\langle\rangle}:i\in I\}$ and $dist^G(1,y)<r$ then $f(y)+l^*,f(y)-l^*\in f(Y)$. (Why? As if $dist^G(1,y^{i}_{\eta})<r$ then $|\eta|<r(i)$, and the construction of \textbf{Step 2}).
\end{itemize}

\textbf{Step 3.} For each $i\in I$ and $\eta\in{}^{<r(i)}2$, define $$B^{i}_{\eta}:=\{x\in B:\text{ there exists a path of length }\leq r(i)\text{ from } x \text{ to } y^{i}_{\eta} \text{ disjoint to }Y\setminus \{y^{i}_{\eta}\}\}$$
and $B^1:=\bigcup_{i\in I,\eta\in{}^{<r(i)}2}B^{i}_{\eta}$.

We now make a few observations:
\begin{itemize}
    \item[$(\alpha)$] If $i_1,i_2\in I$ then, in $G$ there exists a path of length at most $2r$ from $i_1$ to $i_2$ disjoint to $Y$. Why? By the definition of $I$ and (c) in the definition of a saturated forest.
    \item[$(\beta)$] $B^0$ and $B^1$ are disjoint and cover $B$. Why? Trivially they cover $B$, and by $(\alpha)$ and (iv) in the definition of a saturated tree they are disjoint.
    \item[$(\gamma)$] $\langle B^{i}_{\eta}:i\in I,\eta\in{}^{<r(i)}2\rangle$ is a partition of $B^1$. Why? Again trivially they cover $B^1$, and by (iv) in the definition of a saturated tree they are disjoint.
    \item[$(\delta)$] If $\{x,y\}$ is an edge of $G|_B$ then either $x,y\in B^0$, $\{x,y\}=\{i,y^{i}_{\langle\rangle}\}$ for some $i\in I$, $\{x,y\}\subseteq Y$ or $\{x,y\}\subseteq B^{i}_{\eta}$ for some $i\in I$ and $\eta\in{}^{<r(i)}2$. (Use the properties of a saturated forest.)
\end{itemize}

We now define $f|_{B^1}$. Let $\langle (B_j,y_j):j<j^*\rangle$ be some enumeration of $\langle (B^{i}_{\eta},y^{i}_{\eta}):i\in I,\eta\in{}^{<r(i)}2\rangle$. We define $f|_{B_j}$ by induction on $j<j^*$ so assume that $f|_{(\cup_{j'<j}B_{j'})}$ is already defined, and denote: $F=F_j:=f(B^0)\cup f(Y)\cup f(\cup_{j'<j}B_{j'})$. Our construction of $f|_{B_j}$ will satisfy:
\begin{itemize}
    \item $f|_{B_j}$ is one-to-one.
    \item $f(B_j)$ is disjoint to $F_j$.
    \item If $y\in B_j$ then either $f(y)=y$ or $f(y)\not\in[n]$.
\end{itemize}
Let $\langle z^j_s:s<s(j)\rangle$ be some enumeration of the set $\{z\in B_j:G\models y_j\thicksim z\}$. For each $s<s(j)$ choose $l(j,s)$ such that $p_{l(j,s)}>0$ and:
\begin{itemize}
    \item[$\otimes$] If $k\leq 4r$, $(m_1,...,m_k)$ are integers with absolute value not larger than $4r$ and not all equal $0$, and $(s_1,...s_k)$ is a sequence of natural numbers smaller than $j(s)$ without repetitions. Then $|\sum_{1\leq i\leq m}(m_i\cdot l(j,s_i))|>n+\max\{|x|:x\in F_j\}$.
\end{itemize}
Again as $\{l:p_l>0\}$ is infinite we can always choose such $l(j,s)$.
We now define $f|_{B_j}$. For each $y\in B_j$ let $\bar{x}=(x_0,...x_k)$ be a path in $G$ from $y$ to $y_j$, disjoint to $Y\setminus\{y_j\}$, such that $k$ is minimal. So we have $x_0=y$, $x_k=y_j$, $k\leq r$ and $\bar{x}$ is without repetitions. Note that by the definition of $B_j$ such a path exists. For each $0\leq t<k$ define
$$l_t=l_t(\bar{x})
\left\{ \begin{array}{ll}
      l(j,s)&l^{\bar{x}}_t=|y_j-z^j_s| \text{ for some }s<s(j)\\
      -l(j,s)&l^{\bar{x}}_t=-|y_j-z^j_s| \text{ for some }s<s(j)\\
      l^{\bar{x}}_t&\text{otherwise}.
\end{array}\right.$$
Now define $f(y)=f(y_j)+\sum_{0\leq t<k}l_t$. We have to show that $f(y)$ is well defined. Assume that both $\bar{x}_1=(x_0,...x_{k_1})$ and $\bar{x}_2  =(x'_0,...x'_{k_1})$ are paths as above. Then $k_1=k_2$ and $\bar{x}=(x_0,...,x_{k_1},x'_{k_2-1},...,x'_0)$ is a circle of length $k_1+k_2\leq2r$. By (v) in the definition of a saturated tree we know that for each $s<s(j)$, $|y_j-z^j_s|>m_{2r}$. Hence as $\bp$ is without unavoidable circles
we have for each $s<s(j)$ and $0\leq t<k_1+k_2$, if $|l^{\bar{x}}_t|=|y_j-z^j_s|$ then $t\in Sym(\bar{x})$. (see definition \ref{Def:Cycle}(6,7)). Now put for $w\in\{1,2\}$ and $s<s(j)$, $m^+_w(s):=|\{0\leq t<k_w:l^{\bar{x}_w}_t=y_j-z^j_s\}|$ and similarly $m^-_w(s):=|\{0\leq t<k_w:-l^{\bar{x}_w}_t=y_j-z^j_s\}|$. By the definition of $\bar{x}$ we have, $m^+_1(s)-m^-_1(s)=m^+_2(s)-m^-_2(s)$. But from the definition of $l_t(\bar{x})$ we have for $w\in\{1,2\}$, $$\sum_{0\leq t<k_w}l_t(\bar{x}_w)=\sum_{0\leq t<k_w}l^{\bar{x}_w}_t+\sum_{s<s(j)}(m^+_w(s)-m^-_w(s))(l(j,s)-(y_j-z^j_s)).$$ Now as $\sum_{0\leq t<k_1}l^{\bar{x}_1}_t=\sum_{0\leq t<k_2}l^{\bar{x}_2}_t$ we get $\sum_{0\leq t<k_1}l_t(x_1)=\sum_{0\leq t<k_2}l_t(x_2)$ as desired.

We now show that $f|_{B_j}$ is one-to-one. Let $y^1\neq y^2$ be in $B_j$. So for $w\in\{1,2\}$ we have a path $\bar{x}_w=(x^w_0,...x^w_{k_w})$ from $y^w$ to $y_j$. as before, for $s<s(j)$ denote $m^+_w(s):=|\{0\leq t<k_w:l^{\bar{x}_w}_t=y_j-z^j_s\}|$ and similarly $m^-_w(s)$. By the definition of $f_{B_j}$ we have $$f(y^1)-f(y^2)=y^1-y^2+\sum_{s<s(j)}[(m^+_1(s)-m^-_1(s))-(m^+_2(s)-m^-_2(s))]\cdot l(j,s).$$
Now if for each $s<s(j)$, $m^+_1(s)-m^-_1(s)=m^+_2(s)-m^-_2(s)$ then we are done as $y^1\neq y^2$. Otherwise
note that for each $s<s(j)$, $|m^+_1(s)-m^-_1(s)=m^+_2(s)-m^-_2(s)|\leq 4r$. Note further that $|\{s<s(j):m^+_1(s)-m^-_1(s)=m^+_2(s)-m^-_2(s)\neq 0\}|\leq 4r$. Hence by $\otimes$, and as $|y^1-y^2|\leq n$ we are done.

Next let $y\in B_j$ and $\bar{x}=(x_0,...,x_k)$ be a path in $G$ from $y$ to $y_j$. For each $s<s(j)$ define $m^+(s)$ and $m^-(s)$ as above, hence we have $f(y)=y_j+\sum_{s<s(j)}(m^+(s)-m^-(s))l(j,s)$. Consider two cases. First if $(m^+(s)-m^-(s))=0$ for each $s<s(j)$ then $f(y)=y$. Hence $f(y)\not\in f(B^0)=B^0$ (by $(\beta)$ above), $f(y)\not\in f(Y)$ (as $f(Y)\cap[n]=\emptyset$) and $f(y)\not\in f(\cup_{j'<j}B_{j'})$ (by $(\gamma)$ and the induction hypothesis). So $f(y)\not\in F_j$. Second assume that for some $s<s(j)$, $(m^+(s)-m^-(s))\neq0$. Then by the $\otimes$ we have $f(y)\not\in[n]$ and furthermore $f(y)\not\in F_j$. In both cases the demands for $f|_{B_j}$ are met and we are done. After finishing the construction for all $j<j^*$ we have $f|_{B^1}$ such that:
\begin{itemize}
    \item[$\bullet_7$] $f|_{B^1}$ is one-to-one.
    \item[$\bullet_8$] $f({B^1})$ is disjoint to $f(B^0)\cup f(Y)$.
    \item[$\bullet_9$] If $y\in {B^1}$ and $dist^G(1,y)<r$ then $f(y)+l^*,f(y)-l^*\in f({B^1})$. In fact $f(y+l^*)=f(y)+l^*$ and $f(y-l^*)=f(y)-l^*$. (By the construction of Step 3.)
\end{itemize}

Putting $\bullet_1-\bullet_9$ together we have constructed $f:B\to\mathbb{Z}$ that is one-to-one and satisfies:
\begin{itemize}
    \item[$(\circ)$] If $y\in {B}$ and $dist^G(1,y)<r$ then $f(y)+l^*,f(y)-l^*\in f({B})$. Furthermore:
     \item[$(\circ\circ)$]$\{y,f^{-1}(f(y)-l^*)\}$ and $\{y,f^{-1}(f(y)+l^*)\}$ are edges of $G$.
\end{itemize}
For $(\circ\circ)$ use: $\bullet_2$ with the definition of $f|_{B^0}$, $\bullet_5+\bullet'_5$ with the fact that $G\models i\thicksim y^{i}_{\langle\rangle}$, $\bullet_6$ with the construction of Step 2 and $\bullet_9$.

We turn to the definition of $(l,{u_0},U,H)$ and the isomorphism $h:B\to H$. Let $l_{min}=\min\{f(b):b\in B\}$ and $l_{max}=\max\{f(b):b\in B\}$. Define:
\begin{itemize}
    \item $l=l_{min}+l_{max}+1$.
    \item ${u_0}=l_{min}+2$.
    \item $U=\{z+l_{min}+1:z\in Im(f)\}$.
    \item For $b\in B$, $h(b)=f(b)+l_{min}+1$.
    \item For $u,v\in U$, $H\models u\thicksim v$ iff $G\models h^{-1}(u)\thicksim h^{-1}(v)$.
\end{itemize}
As $f$ was one-to-one so is $h$, and trivially it is onto $U$ and maps $1$ to ${u_0}$. Also by the definition of $H$, $h$ is a graph isomorphism. So it remains to show that $(l,{u_0},U,H)$ is $r$-proper. First $(*)_1$ in the definition of proper is immediate from the definition of $H$. Second for $(*)_2$ in the definition of proper let $u\in U$ be such that $dist^H({u_0},u)<r$. Denote $y:=h^{-1}(u)$ then by the definition of $H$ we have $dist^G(1,y)<r$, hence by $(\circ)$, $f(y)+l^*,f(y)-l^*\in f(B)$ and hence by the definition of $h$ and $U$, $u+l^*, u-l^* \in U$ as desired. Lastly to see $(*)_3$ let $u,u'\in U$ and denote $y=h^{-1}(u)$ and $y'=h^{-1}(u')$. Assume $|u-u'|=l^*$ then by $(\circ\circ)$ we have $G\models y\thicksim y'$ and by the definition of $H$, $H\models u\thicksim u'$. Now assume that $H\models u\thicksim u'$ then $G\models y\thicksim y'$. Using observation $(\delta)$ above and rereading  1-3 we see that $|u-u'|$ is either $l^*$, $|y-y'|$, $l_{\eta}$ for some $\eta\in{}^{< r}2$ (see Step 2) or $l(j,s)$ for some $j<j^*,s<s(j)$ (see step 3). In all cases we have $P_{|u-u'|}>0$. Together we have $(*)_3$ as desired. This completes the proof of Theorem \ref{ThmOne1}.
\end{proof}

\end{document}